\documentclass[11pt]{article}
\usepackage{geometry} 
\usepackage[utf8]{inputenc} 
\usepackage[T1]{fontenc}    
\usepackage{hyperref}       
\usepackage{url}            
\usepackage{booktabs}       
\usepackage{amsfonts}       
\usepackage{nicefrac}       
\usepackage{microtype}      
\usepackage{algorithm,algorithmic}
\usepackage{datetime}
\usepackage{amsthm}
\usepackage{color}
\usepackage{mathrsfs}
\usepackage{amsmath}
\usepackage{amssymb}
\usepackage{graphicx}
\usepackage{makecell}
\usepackage{bm}
\usepackage{subfigure}
\usepackage{natbib}
\usepackage{mathtools}

\usepackage[dvipsnames]{xcolor}
\usepackage{pifont}
\newcommand{\cmark}{{\color{PineGreen}\ding{51}}}%
\newcommand{\xmark}{{\color{BrickRed}\ding{55}}}%
\newcommand{\yesmark}{\textrm{\color{PineGreen} YES}}%
\newcommand{\nomark}{\textrm{\color{BrickRed} NO}}%

\newcommand{\makecellnew}[1]{{\renewcommand{\arraystretch}{1.2}\begin{tabular}{c} #1 \end{tabular}}}
\usepackage{colortbl}
\definecolor{bgcolor}{rgb}{0.65, 0.85, 0.99}

\newcommand{\eqdef}{\coloneqq}
\newcommand{\cO}{{\cal O}}
\newcommand{\R}{\mathbb{R}}
\newcommand{\E}{\mathbb{E}}

\def\Pr{{\rm Prob}}

\def\erf{{\rm erf}}

\DeclareMathOperator*{\argmin}{arg\,min}
\DeclareMathOperator*{\sign}{{\rm sign}}
\DeclareMathOperator*{\var}{{\rm Var}}

\def\<{\left\langle}
\def\>{\right\rangle}
\def\[{\left[}
\def\]{\right]}
\def\({\left(}
\def\){\right)}

\theoremstyle{plain}
\newtheorem{theorem}{Theorem}  
\newtheorem{lemma}[theorem]{Lemma} 
\newtheorem{remark}{Remark} 

\theoremstyle{definition}
\newtheorem{assumption}{Assumption} 
\newtheorem{definition}{Definition}

\theoremstyle{remark}

\graphicspath{{/}}

\begin{document}

\title{\bf Stochastic Sign Descent Methods: \\ New Algorithms and Better Theory}

\author{Mher Safaryan$^1$ \qquad Peter Richt\'arik$^{1,2}$ \\
\phantom{XXX} \\
 $^1$King Abdullah University of Science and Technology (KAUST), Thuwal, Saudi Arabia \\
 $^2$Moscow Institute of Physics and Technology (MIPT), Moscow, Russia
}

\date{}
\maketitle

\begin{abstract}
Various gradient compression schemes have been proposed to mitigate the communication cost in distributed training of large scale machine learning models. Sign-based methods, such as signSGD \citep{BWAA}, have recently been gaining popularity because of their simple compression rule and connection to adaptive gradient methods, like ADAM.
In this paper, we analyze sign-based methods for non-convex optimization in two key settings:
(i) the standard single-node setting, and
(ii) the parallel shared-data or distributed homogeneous setting.
For single machine case, we generalize the previous analysis of signSGD relying on intuitive bounds on success probabilities and allowing even biased estimators.
Furthermore, we extend the analysis to parallel setting within a parameter server framework, where exponentially fast noise reduction is guaranteed with respect to number of nodes, maintaining $1$-bit compression in both directions and using small mini-batch sizes.
Next, we establish a fundamental limitation of sign-based algorithms under heterogeneous local objectives. In the homogeneous setting, we introduce a new sign-based method, {\em Stochastic Sign Descent with Momentum (SSDM)}, and show that, under the standard bounded-variance assumption, it achieves the optimal asymptotic convergence rate in the standard Euclidean norm of the gradient
\footnote{\textcolor{red}{Post-publication correction: The original statement of Theorem~\ref{thm:ssdm} allowed fully heterogeneous local objectives, while its proof implicitly required every worker's stochastic gradient to be unbiased for the same global objective. The corrected result here is therefore restricted to the shared-data or homogeneous setting.}}.
We validate several aspects of our theoretical findings with numerical experiments.
\end{abstract}

\section{Introduction}

One of the key factors behind the success of modern machine learning models is the availability of large amounts of training data \citep{BoLe, KSH, Sch}. However, the state-of-the-art deep learning models deployed in industry typically rely on datasets too large to fit the memory of a single computer, and hence the training data is typically split and stored across a number of compute nodes capable of working in parallel.   Training such models then amounts to solving optimization problems of the form
\begin{equation} \label{eq:finite_sum}
\min_{x\in \R^d} f(x) \eqdef \frac{1}{M} \sum \limits_{n=1}^M f_n(x),
\end{equation}
where $f_n:\R^d\to \R$ represents the {\em non-convex} loss of a deep learning model parameterized by $x\in \R^d$ associated with data stored on node $n$. Arguably, stochastic gradient descent (SGD)~\citep{RoMo, VBS, QRGSLS} in of its many variants \citep{KiBa, DHS, SRB, Zei, GhLa} is the most popular algorithm for solving \eqref{eq:finite_sum}. In its basic implementation, all  workers $n\in \{1,2,\dots,M\}$ in parallel compute a random approximation $\hat{g}^n(x_k)$ of $\nabla f_n(x_k)$, known as the {\em stochastic gradient}. These approximations are then sent to a master node which performs the aggregation 
$
\hat{g}(x_k) \eqdef \frac{1}{M}\sum_{n=1}^M \hat{g}^n(x_k).$ The aggregated vector is subsequently broadcast back to the nodes, each of which  performs an update of the form
\begin{equation*}
x_{k+1} = x_k - \gamma_k \hat{g}(x_k),
\end{equation*}
updating their local copies of the parameters of the model. 


{ 
    
    \begin{table*}[t]
        \centering
        \scriptsize
        \addtolength{\tabcolsep}{-3pt}
        \caption{Summary of main theoretical results obtained in this work.}
        \label{tbl:summary}
        \vskip 0.15in
        \renewcommand{\arraystretch}{1.7}
        \begin{tabular}{|c|c|c|c|c|c|c|c|}
        \hline
        \makecell{}
            & \makecell{Convergence \\ rate}
            & \makecell{Gradient \\ norm used \\ in theory}
            & \makecell{Weak noise \\ assumptions}
            & \makecell{Weak \\ dependence \\ on smoothness}
            & \makecell{Can handle \\ biased \\ estimator?}
            & \makecell{Can work \\ with small \\ minibatch?}
            & \makecell{Can handle \\ partitioned \\ train data?} \\
        \hline
            \hline
            \makecell{SGD \\ \citep{GhLa}}
            & \makecell{$\cO\left(\frac{1}{\sqrt{K}}\right)$} 
            & \makecell{$l^2$ norm \\ squared} 
            & \makecell{$\textrm{Var}[\hat{g}]\le\sigma^2$} 
            & \makecell{\xmark \; $\max\limits_{i=1}^d L_i$} 
            & \makecell{\nomark} 
            & \makecell{\yesmark} 
            & \makecell{\yesmark} 
            \\
            \hline
            \makecell{signSGD \\ \citep{BZAA}}
            & \makecell{$\cO\left(\frac{1}{\sqrt{K}}\right)$} 
            & \makecell{a mix of \\ $l^1$ and $l^2$ \\ squared} 
            & \makecell{\xmark  \\ unimodal, \\ symmetric \& \\ $\textrm{Var}[\hat{g}_i]\le\sigma_i^2$} 
            & \makecell{\cmark \; $\frac{1}{d}\sum \limits_{i=1}^d L_i$} 
            & \makecell{\nomark} 
            & \makecell{\yesmark} 
            & \makecell{\nomark} 
            \\
            \hline
            \makecell{signSGD \\ with $M$ Maj.Vote \\ \citep{BZAA}}
            & \makecell{$\cO\left(\frac{1}{K^{\nicefrac{1}{4}}}\right)$ \\[12pt] (\text{\small speedup}$\sim\tfrac{1}{\sqrt{M}}$)} 
            & \makecell{$l^1$ norm} 
            & \makecell{\xmark  \\ unimodal, \\ symmetric \& \\ $\textrm{Var}[\hat{g}_i]\le\sigma_i^2$} 
            & \makecell{\cmark \; $\frac{1}{d}\sum \limits_{i=1}^d L_i$} 
            & \makecell{\nomark} 
            & \makecell{\nomark} 
            & \makecell{\nomark} 
            \\
            \hline
            \makecell{Signum \\ \citep{BWAA}}
            & \makecell{$\cO\left(\frac{\log K}{K^{\nicefrac{1}{4}}}\right)$} 
            & \makecell{$l^1$ norm} 
            & \makecell{\xmark  \\ unimodal, \\ symmetric \& \\ $\textrm{Var}[\hat{g}_i]\le\sigma_i^2$} 
            & \makecell{\cmark \; $\frac{1}{d}\sum \limits_{i=1}^d L_i$} 
            & \makecell{\nomark} 
            & \makecell{\nomark} 
            & \makecell{\nomark} 
            \\
            \hline
            \makecell{Noisy signSGD \\ \citep{noisySignSGD}}
            & \makecell{$\cO\left(\frac{1}{K^{\nicefrac{1}{4}}}\right)$} 
            & \makecell{a mix of \\ $l^1$ and $l^2$ \\ squared} 
            & \makecell{\xmark  \\ absence of noise \\ $\hat{g} = \nabla f$} 
            & \makecell{\xmark \; $\max\limits_{n=1}^M L^n$} 
            & \makecell{\nomark} 
            & \makecell{\nomark} 
            & \makecell{\yesmark} 
            \\
            \hline
            \hline
            \rowcolor{bgcolor}
            \makecellnew{signSGD \\ {\bf This work} (Thm.~\ref{non-convex-theorem}, \ref{non-convex-theorem-2})}
            & \makecellnew{$\cO\left(\frac{1}{\sqrt{K}}\right)$} 
            & \makecellnew{$\rho$-norm} 
            & \makecellnew{\cmark  \\ $\rho_i > \frac{1}{2}$ } 
            & \makecellnew{\cmark \; $\frac{1}{d}\sum \limits_{i=1}^d L_i$} 
            & \makecellnew{\yesmark} 
            & \makecellnew{\yesmark} 
            & \makecellnew{\nomark} 
            \\
            \hline
            \rowcolor{bgcolor}
            \makecellnew{signSGD \\ with $M$ Maj.Vote \\ {\bf This work} (Thm.~\ref{non-convex-theorem-d})}
            & \makecellnew{$\cO\left(\frac{1}{\sqrt{K}}\right)$\\[5pt] (\text{\small speedup}$\sim e^{-M}$)} 
            & \makecellnew{$\rho_M$-norm} 
            & \makecellnew{\cmark  \\ $\rho_i > \frac{1}{2}$ } 
            & \makecellnew{\cmark \; $\frac{1}{d}\sum \limits_{i=1}^d L_i$} 
            & \makecellnew{\yesmark} 
            & \makecellnew{\yesmark} 
            & \makecellnew{\nomark} 
            \\
            \hline
            \rowcolor{bgcolor}
            \makecellnew{SSDM {\sc (Alg.~\ref{alg:ssdm})} \\ {\bf This work} (Thm.~\ref{thm:ssdm})}
            & \makecellnew{$\cO\left(\frac{1}{K^{\nicefrac{1}{4}}}\right)$} 
            & \makecellnew{$l^2$ norm} 
            & \makecellnew{$\textrm{Var}[\hat{g}]\le\sigma^2$} 
            & \makecellnew{\cmark \; $\frac{1}{M}\sum \limits_{n=1}^M L^n$} 
            & \makecellnew{\nomark} 
            & \makecellnew{\yesmark} 
            & \makecellnew{\nomark} 
            \\
        \hline
        \end{tabular}   
        \vskip -0.1in    
    \end{table*}
}

\subsection{Gradient compression}  
Typically, communication of the local gradient estimators $\hat{g}^n(x_k)$ to the master forms the bottleneck of such a system \citep{SFDLY, ZLKALZ, LHMWD}. In an attempt to alleviate this communication bottleneck, a number of compression schemes for gradient updates  have been proposed and analyzed \citep{AGLTV, WSLCPW, WXYWWCL, KFJ, MGTR}. A {\em compression scheme} is a (possibly randomized)  mapping $Q:\R^d\to \R^d$,  applied by the nodes to $\hat{g}^n(x_k)$ (and possibly also by the master to aggregated update in situations when broadcasting is expensive as well) in order to reduce the number of bits of the communicated message. 

\textbf{Sign-based compression.}  Although most of the existing theory is limited to {\em unbiased} compression schemes, i.e., $\E Q(x) = x$,  {\em biased} schemes such as those based on communicating signs of the update entries only often perform much better~\citep{SFDLY, Strom, WXYWWCL, CCC, BaHe, BWAA, BZAA, ZRSKK, LCCH}.  The simplest  among these sign-based  methods is signSGD (see Algorithm~\ref{alg:signSGD}), whose update direction is assembled from the component-wise signs of the stochastic gradient. 

\textbf{Adaptive methods.} While ADAM is one of the most  popular {\em adaptive} optimization methods used in deep learning \citep{KiBa}, there are issues with its convergence \citep{RKK} and generalization \citep{WRSSR} properties. It was noted by \citet{BaHe} that the behaviour of ADAM is similar to a momentum version of signSGD. Connection between sign-based and adaptive methods has long history, originating at least in Rprop~\citep{RiBr} and RMSprop~\citep{TiHi}. Therefore, investigating the behavior of signSGD can improve our understanding on the convergence of adaptive methods such as ADAM. 


\section{Contributions}

We now present the main contributions of this work. Our key results are summarized in Table \ref{tbl:summary}.

\subsection{Single machine setup}

$\bullet$ {\bf 2 methods for 1-node setup.}   In   the $M=1$ case, we study two general classes of sign based methods for minimizing a smooth non-convex function $f$. The first method has the standard form\footnotemark
\begin{equation}\label{eq:method1}x_{k+1} = x_k - \gamma_k\sign\hat{g}(x_k),\end{equation}
while the second has a new form not considered in the literature before:
\begin{equation}\label{eq:method2}x_{k+1} = \argmin\{f(x_k), f(x_k - \gamma_k\sign\hat{g}(x_k))\}.\end{equation}

\footnotetext{$\sign g$ is applied element-wise to the entries $g_1,g_2,\dots,g_d$ of $g\in \R^d$. For $t\in \R$ we define $\sign t = 1$ if $t>0$, $\sign t = 0$ if $t=0$,  and $\sign t = -1$ if $t<0$.}

\noindent $\bullet$ {\bf Key novelty.}  The key novelty of our methods is in a {\em substantial relaxation} of the requirements that need to be imposed on the gradient estimator $\hat{g}(x_k)$ of the true gradient $\nabla f(x^k)$. In sharp contrast with existing approaches, we allow $\hat{g}(x_k)$ to be {\em biased}. Remarkably, we only need one additional and rather weak assumption on $\hat{g}(x_k)$ for the methods to provably converge: we  require the signs of the entries of $\hat{g}(x_k)$ to be equal to the signs  of the entries of $g(x^k) \eqdef \nabla f(x^k)$ with a probability strictly larger than $\nicefrac{1}{2}$ (see Assumption~\ref{asum-SPB}). Formally, we assume the following bounds on success probabilities:
\begin{equation*}\tag{SPB}
\Pr(\sign\hat{g}_i(x_k) = \sign g_i(x_k) ) > \frac{1}{2}
\end{equation*}
for all $i \in \{1,2,\dots,d\}$ with $g_i(x_k) \ne 0$.

We provide three necessary conditions for our assumption to hold (see Lemma \ref{lemma:main-rho}, \ref{lemma:strict-rho} and \ref{lemma:SPB-via-minibatch}) and show through a counterexample that a slight violation of this assumption breaks the convergence.

\noindent $\bullet$ {\bf Convergence theory.} While our complexity bounds have the same $\cO(\nicefrac{1}{\sqrt{K}})$ dependence on the number of iterations, they have a {\em better dependence on the smoothness parameters} associated with $f$. Theorem~\ref{non-convex-theorem} is the  first result on signSGD for non-convex functions which does not rely on mini-batching, and which allows for step sizes independent of the total number of iterations $K$. Finally, Theorem 1 in \citep{BZAA} can be recovered from our general Theorem~\ref{non-convex-theorem}. Our bounds are cast in terms of a {\em novel norm-like function, which we call the $\rho$-norm}, which  is a weighted $l^1$ norm with positive  variable weights. 

\noindent $\bullet$ {\bf Multi-node case with noise reduction at exponential speed.} Under the same SPB assumption, we extend our results to the {\em parallel setting} with arbitrary $M$ nodes, where we also consider sign-based compression of the aggregated gradients. Considering the noise-free training as a baseline, we guarantee exponentially fast noise reduction with respect to $M$ (see Theorem \ref{non-convex-theorem-d}).

\subsection{Distributed setup}

\noindent $\bullet$ {\bf Limitation under heterogeneous local objectives.} We describe a fundamental scaling obstacle that prevents sign-based aggregation from guaranteeing convergence for arbitrary heterogeneous settings.

\noindent $\bullet$ {\bf New sign-based method for homogeneous case.} To improve the convergence theory in the homogeneous setting, we propose a new distributed sign-based method--{\em Stochastic Sign Descent with Momentum (SSDM)}; see Algorithm \ref{alg:ssdm}.

\noindent $\bullet$ {\bf Key novelty.} The key novelty in our SSDM method is the notion of {\em stochastic sign} operator $\widetilde{\sign}:\R^d\to\R^d$ defined as follows:
\begin{equation*}
\left(\widetilde{\sign}\,g\right)_i =
\begin{cases}
+1, & \text{ with probability } \frac{1}{2} + \frac{1}{2}\frac{g_i}{\|g\|} \\
-1, & \text{ with probability } \frac{1}{2} - \frac{1}{2}\frac{g_i}{\|g\|}
\end{cases}
\end{equation*}
for all $i \in \{1,2,\dots,d\}$ and $\widetilde{\sign}\,\mathbf{0} = \mathbf{0}$ with probability $1$.
Unlike the deterministic sign operator, stochastic $\widetilde{\sign}$ naturally satisfies the SPB assumption and it gives an unbiased estimator with a proper scaling factor.

\noindent $\bullet$ {\bf Convergence theory.}
Under the standard bounded variance condition, our SSDM method guarantees the optimal asymptotic rate $\cO(\varepsilon^{-4})$ without {\em error feedback} trick and communicating sign-bits only (see Theorem \ref{thm:ssdm}).

\section{Success Probabilities and Gradient Noise} \label{sec:success_probs} 

In this section we describe our key (and weak) assumption on the gradient estimator $\hat{g}(x)$, and show through a counterexample that without this assumption, signSGD can fail to converge.

\subsection{Success probability bounds}

\begin{assumption}[SPB: Success Probability Bounds]\label{asum-SPB}
For any  $x\in\R^d$, we have access to an independent (and {\em not necessarily unbiased}) estimator $\hat{g}(x)$ of the true gradient $g(x)\eqdef \nabla f(x)$ that if $g_i(x)\ne0$, then
\begin{equation}\label{SP-bounds}
\rho_i(x) \eqdef \Pr \left(\sign\hat{g}_i(x) = \sign g_i(x) \right) > \frac{1}{2}
\end{equation}
for all $x\in \R^d$ and all $i \in \{1,2,\dots,d\}$.
\end{assumption}

We will refer to the probabilities $\rho_i$ as {\em success probabilities}. As we will see, they play a central role in the convergence of sign based methods.
Moreover, we argue that it is  reasonable to require from the sign of stochastic gradient to show true gradient direction more likely than the opposite one. Extreme cases of this assumption are the absence of gradient noise, in which case $\rho_i=1$, and an overly noisy stochastic gradient, in which case $\rho_i\approx\frac{1}{2}$.

\begin{remark}
Assumption \ref{asum-SPB} can be relaxed by replacing bounds (\ref{SP-bounds}) with
\begin{equation*}
\E\left[\sign\left(\hat{g}_i(x)\cdot g_i(x)\right)\right] > 0, \quad\text{if}\quad g_i(x)\ne0.
\end{equation*}
However, if $\sign\hat{g}_i(x) \ne 0$ almost surely (e.g. $\hat{g}_i(x)$ has continuous distribution), then these bounds are identical.
\end{remark}
\textbf{Extension to stochastic sign oracle.} Notice that we do \emph{not} require $\hat{g}$ to be unbiased and we do \emph{not} assume uniform boundedness of the variance, or of the second moment. This observation allows to extend existing theory to more general sign-based methods with a stochastic sign oracle. By a stochastic sign oracle we mean an oracle that takes $x_k\in \R^d$ as an input, and outputs a random vector $\hat{s}_k\in \R^d$ with entries in $\pm 1$. However, for the sake of simplicity, in the rest of the paper we will work with the signSGD formulation, i.e., we let $\hat{s}_k = \sign\hat{g}(x_k)$.

\subsection{A counterexample to signSGD} \label{sec:counterexample}
Here we analyze a counterexample to signSGD discussed in \citep{SQSM}. Consider the following least-squares problem with unique minimizer $x^*=(0,0)$:

\begin{equation}\label{counter-ex-1node}
\min_{x\in\R^2} f(x) = \frac{1}{2}\left[\langle a_1, x\rangle^2 + \langle a_2, x\rangle^2\right], \qquad
a_1 = \left[\begin{smallmatrix} 1+\varepsilon \\ -1+\varepsilon\end{smallmatrix}\right], \;
a_2 = \left[\begin{smallmatrix} -1+\varepsilon \\ 1+\varepsilon\end{smallmatrix}\right],
\end{equation}
where $\varepsilon\in(0,1)$ and stochastic gradient $\hat{g}(x) = \nabla \langle a_i, x\rangle^2 = 2\langle a_i, x\rangle a_i$ with probabilities $\nicefrac{1}{2}$ for $i=1,2$. Let us take any point from the line $Z=\{(z_1,z_2)\colon z_1+z_2=2\}$ as initial point $x_0$ for the algorithm and notice that $\sign\hat{g}(x) = \pm(1,-1)$ for any $x\in Z$. Hence, signSGD with any step-size sequence remains stuck along the line $Z$, whereas the problem has a unique minimizer at the origin.

In this example, Assumption~\ref{asum-SPB} is violated. Indeed, notice that
$
\sign\hat{g}(x) = (-1)^i\sign\langle a_i, x\rangle 
\left[\begin{smallmatrix} -1 \\ 1\end{smallmatrix}\right]
$
with probabilities $\nicefrac{1}{2}$ for $i = 1, 2$. By $S \eqdef \{x\in\R^2\colon \langle a_1, x\rangle \cdot \langle a_2, x\rangle > 0\}\neq\emptyset$ denote the open cone of points having either an acute or an obtuse angle with both $a_i$'s. Then for any $x\in S$, the sign of the stochastic gradient is $\pm(1,-1)$ with probabilities $\nicefrac{1}{2}$. Hence for any $x\in S$, we have low success probabilities:
$$
\rho_i(x) = \Pr\left(\sign\hat{g}_i(x) = \sign g_i(x) \right) \le \frac{1}{2}, \; i = 1, 2.
$$
So, in this case we have an entire conic region with low success probabilities, which clearly violates (\ref{SP-bounds}). Furthermore, if we take a point from the complement open cone $\bar{S}^c$, then the sign of stochastic gradient equals to the sign of gradient, which is perpendicular to the axis of $S$ (thus in the next step of the iteration we get closer to $S$). For example, if $\langle a_1, x\rangle < 0$ and $\langle a_2, x\rangle > 0$, then $\sign\hat{g}(x) = (1,-1)$ with probability 1, in which case $x-\gamma\sign\hat{g}(x)$ gets closer to low success probability region $S$.

In summary, in this counterexample there is a conic region where the sign of the stochastic gradient is useless (or behaves adversarially), and for any point outside that region, moving direction (which is the opposite of the sign of gradient) leads toward that conic region.

\subsection{Sufficient conditions for SPB}
To motivate our SPB assumption, we compare it with 4 different conditions commonly used in the literature and show that it holds under general assumptions on gradient noise. Below, we assume that for any point $x\in\R^d$, we have access to an independent and unbiased estimator $\hat{g}(x)$ of the true gradient $g(x)=\nabla f(x)$.
%
\begin{lemma}[see \ref{apx:lemma-main-rho}]\label{lemma:main-rho}
If for each coordinate $\hat{g}_i$ has a unimodal and symmetric distribution with variance $\sigma_i^2=\sigma_i^2(x)$, $1\le i\le d$ and $g_i\ne0$, then
$$
\rho_i \ge \frac{1}{2} + \frac{1}{2}\frac{|g_i|}{|g_i| + \sqrt{3}\sigma_i} > \frac{1}{2}.
$$
\end{lemma}
This is the setup used in Theorem 1 of \citet{BZAA}. We recover their result as a special case using Lemma \ref{lemma:main-rho} (see Appendix~\ref{sec:recovering}).
Next, we replace the distribution condition by coordinate-wise strong growth condition (SGC) \citep{SR, VBS} and fixed mini-batch size.

\begin{lemma}[see \ref{apx:lemma-strict-rho}]\label{lemma:strict-rho}
Let coordinate-wise variances $\sigma_i^2(x) \le c_i\, g_i^2(x)$ are bounded for some constants $c_i\ge 0$. Choose  mini-batch size $\tau>2\max_i c_i$ for stochastic gradient estimator. If further $g_i\ne0$, then
$$
\rho_i \ge 1 - \frac{c_i}{\tau} > \frac{1}{2}.
$$
\end{lemma}
Now we remove SGC and give an adaptive condition on mini-batch size of stochastic gradient for the SPB assumption to hold.
\begin{lemma}[see \ref{apx:lemma-SPB-via-minibatch}]\label{lemma:SPB-via-minibatch}
Let $\sigma_i^2=\sigma_i^2(x)$ be the variance and $\nu_i^3=\nu_i^3(x)$ be the 3th central moment of $\hat{g}_i(x)$, $1\le i\le d$. Then SPB assumption holds if mini-batch size
$$
\tau > 2\min \left( \frac{\sigma_i^2}{g_i^2}, \frac{\nu_i^3}{|g_i|\sigma_i^2}\right).
$$
\end{lemma}

Finally, we compare SPB with the standard bounded variance assumption in the sense of differential entropy.

\begin{lemma}[see \ref{apx:lemma-SPB-via-entropy}]\label{lemma:SPB-via-entropy}
Differential entropy of a probability distribution under the bounded variance assumption is bounded, while under the SPB assumption it could be arbitrarily large.
\end{lemma}

Differential entropy argument is an attempt to bridge our new SPB assumption to one of the most basic assumptions in the literature, bounded variance assumption. Clearly, they are not comparable in the usual sense, and neither one is implied by the other. Still, we propose another viewpoint to the situation and compare such conditions through the lens of information theory. Practical meaning of such observation is that SPB handles a much broader (though not necessarily more important) class of gradient noise than bounded variance condition. In other words, this gives an intuitive measure on how much restriction we put on the noise.

Note that SPB assumption describes the convergence of sign descent methods, which is known to be problem dependent (e.g. see \citep{BaHe}, section 6.2 Results). One should view the SPB condition as a criteria to problems where sign based methods are useful.

\section{A New ``Norm'' for Measuring the Gradients} \label{sec:rho-norm}

In this section we introduce the concept of a norm-like function, which we call {\em $\rho$-norm},  induced from success probabilities. Used to measure gradients in our convergence rates, $\rho$-norm  is a technical tool enabling the analysis. 

\begin{figure}[t]
\vskip 0.2in
\begin{center}
\centerline{\includegraphics[scale=0.5]{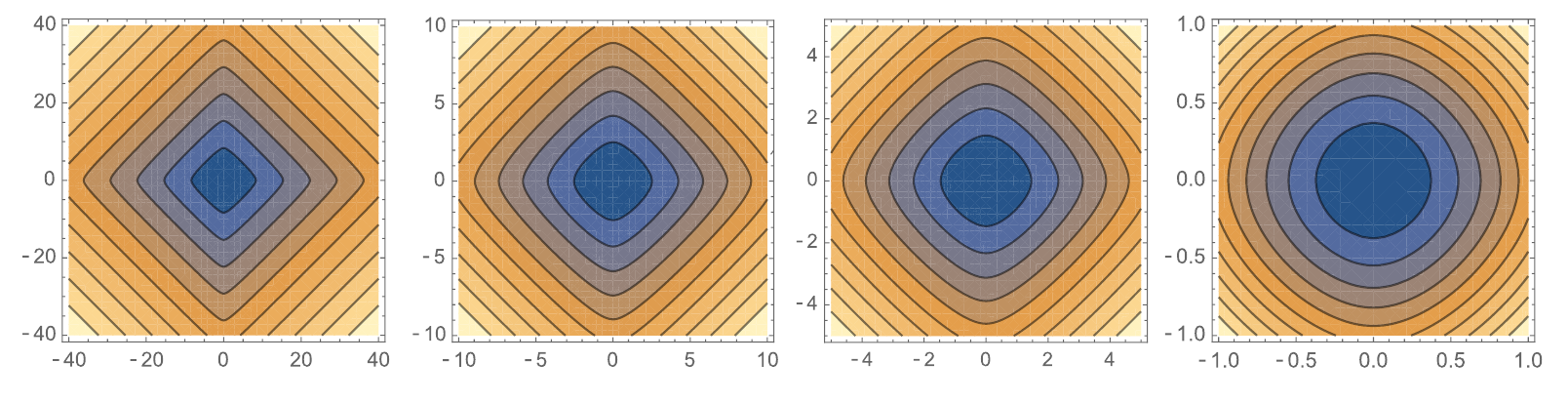}}
\caption{Contour plots of the $l^{1,2}$ norm (\ref{def:l122-norm}) at 4 different scales with fixed noise $\sigma = 1$.}
\label{fig:contours}
\end{center}
\vskip -0.2in
\end{figure}

\begin{definition}[$\rho$-norm]\label{def:rho-norm}
Let $\rho \eqdef \{\rho_i(x)\}_{i=1}^d$ be the collection of probability functions from the SPB assumption. We define  the $\rho$-norm of gradient $g(x)$ via
$$
\|g(x)\|_\rho \eqdef \sum_{i=1}^d (2\rho_i(x)-1)|g_i(x)|.
$$
\end{definition}

Note that mixed $\rho$-norm is not a norm as it may not satisfy the triangle inequality. However, under SPB assumption, it is positive definite as it is a weighted $l^1$ norm with positive (and variable) weights $2\rho_i(x)-1 > 0$. That is, $\|g\|_\rho \ge 0$,  and $\|g\|_\rho = 0$  if and only if $g = 0$.
Although, in general, $\rho$-norm is not a norm in classical sense, it can be reduced to one in special cases. For example, under the assumptions of Lemma~\ref{lemma:main-rho}, $\rho$-norm can be lower bounded by a mixture of the $l^1$ and squared $l^2$ norms:
\begin{equation}\label{def:l122-norm}
\|g\|_\rho = \sum \limits_{i=1}^d (2\rho_i-1)|g_i| \ge \sum \limits_{i=1}^d \frac{g_i^2}{|g_i| + \sqrt{3}\sigma_i} \eqdef \|g\|_{l^{1,2}}.
\end{equation}
Note that $l^{1,2}$-norm is again not a norm. However, it is positive definite, continuous and order preserving, i.e., for any $g^k$, $g$, $\tilde{g}\in\R^d$ we have:
\begin{enumerate}
\item $\|g\|_{l^{1,2}} \ge 0$ and $\|g\|_{l^{1,2}} = 0$ if and only if $g = 0$,
\item $g^k\to g$ (in $l^2$ sense) implies $\|g^k\|_{l^{1,2}} \to \|g\|_{l^{1,2}}$,
\item $0\le g_i\le \tilde{g}_i$ for any $1\le i\le d$ implies $\|g\|_{l^{1,2}} \le \|\tilde{g}\|_{l^{1,2}}$.
\end{enumerate}

From these three properties it follows that $\|g^k\|_{l^{1,2}}\to 0$ implies $g^k\to 0$. These properties  are important as we will measure convergence rate in terms of the $l^{1,2}$ norm in the case of unimodal and symmetric noise assumption.
To understand the nature of the $l^{1,2}$ norm, consider the following two cases when $\sigma_i(x)\le c|g_i(x)| + \tilde{c}$ for some constants $c,\,\tilde{c}\ge0$. If the iterations are in $\varepsilon$-neighbourhood of a minimizer $x^*$ with respect to  the $l^{\infty}$ norm (i.e., $\max_{1\le i\le d}|g_i|\le\varepsilon$), then the $l^{1,2}$ norm is equivalent to scaled $l^2$ norm squared:
$$
\frac{1}{\left(1 + \sqrt{3}c\right)\varepsilon + \sqrt{3}\tilde{c}} \|g\|_2^2 \le \|g\|_{l^{1,2}} \le \frac{1}{\sqrt{3}\tilde{c}} \|g\|_2^2 .
$$
On the other hand, if iterations are away from a minimizer (i.e., $\min_{1\le i\le d}|g_i|\ge L$), then the $l^{1,2}$-norm is equivalent to scaled $l^1$ norm:
$$
\frac{1}{1 + \sqrt{3}(c + \tilde{c}/L)} \|g\|_1 \le \|g\|_{l^{1,2}} \le \frac{1}{1 + \sqrt{3}c} \|g\|_1 .
$$

These equivalences are visible in Figure~\ref{fig:contours}, where we plot the level sets of $g\mapsto \|g\|_{l^{1,2}}$ at various distances from the origin.
Similar mixed norm observation for signSGD was also noted by \citet{BZAA}.
Alternatively, under the assumptions of Lemma~\ref{lemma:strict-rho}, $\rho$-norm can be lower bounded by a weighted $l^1$ norm with positive constant weights $1-\nicefrac{2c_i}{\tau}>0$:
$$
\|g\|_\rho = \sum_{i=1}^d (2\rho_i-1)|g_i| \ge \sum_{i=1}^d (1-\nicefrac{2c_i}{\tau})|g_i|.
$$

\section{Convergence Theory}

Now we turn to our theoretical results of sign based methods. First we give our general convergence rates under the SPB assumption. Afterwards, we extend the theory to parallel setting under the same SPB assumptions with majority vote aggregation. Finally, we explain the limitation of sign-based aggregation for heterogeneous local objectives and analyze a new sign based method, {\em SSDM}, in the homogeneous setting.

\begin{algorithm}[H]
\begin{algorithmic}[1]
\STATE \textbf{Input:} step size $\gamma_k$, current point $x_k$
\STATE $\hat{g}_k \leftarrow \textrm{StochasticGradient}(f, x_k)$
\STATE {\em Option~1:} $x_{k+1} = x_k - \gamma_k\sign\hat{g}_k$
\STATE {\em Option~2:} $x_{k+1} = \argmin\{f(x_k), f(x_k - \gamma_k\sign\hat{g}_k)\}$
\end{algorithmic}
\caption{\sc signSGD}
\label{alg:signSGD}
\end{algorithm}

Throughout the paper we assume that function $f\colon\R^d\to\R$ is nonconvex and lower bounded, i.e., $f(x)\ge f^*$ for all $x\in\R^d$.

\subsection{Convergence analysis for $M=1$}
We start our convergence theory with single node setting, where $f$ is smooth with some non-negative constants $(L_i)_{i=1}^d$, i.e.,
$$
f(y) \le f(x) + \langle \nabla f(x), y-x\rangle + \sum_{i=1}^d \frac{L_i}{2}(y_i-x_i)^2
$$
for all $x,\,y\in\R^d$. Let $\bar L \eqdef \frac{1}{d}\sum_{i=1}^d L_i$. 

We now state our convergence result for signSGD (\ref{eq:method1}) under the general SPB assumption.
\begin{theorem}[see \ref{apx:non-convex-theorem}]\label{non-convex-theorem}
Under the SPB assumption, single node signSGD (Algorithm \ref{alg:signSGD}) with Option 1 and with step sizes $\gamma_k = \gamma_0/\sqrt{k+1}$ converges as follows
\begin{equation}\label{non-convex-rate-min}
\min\limits_{0\le k<K} \E\|\nabla f(x_k)\|_\rho \le \frac{f(x_0) - f^*}{\gamma_0\sqrt{K}} + \frac{3\gamma_0 d\bar{L}}{2}\frac{\log{K}}{\sqrt{K}} \;.
\end{equation}
If $\gamma_k \equiv \gamma>0$, we get $\nicefrac{1}{K}$ convergence to a neighbourhood:
\begin{equation}\label{const-step-rate}
\frac{1}{K}\sum_{k=0}^{K-1}\E\|\nabla f(x_k)\|_\rho \le \frac{f(x_0) - f^*}{\gamma K} +\frac{ \gamma d\bar{L}}{2} \;.
\end{equation}
\end{theorem}

We now comment on the above result:

$\bullet$ \textbf{Generalization.} Theorem \ref{non-convex-theorem} is the  first general result on signSGD for non-convex functions without mini-batching, and with step sizes independent of the total number of iterations $K$. Known convergence results \citep{BWAA, BZAA} on signSGD use mini-batches and/or step sizes dependent on $K$. Moreover, they also use unbiasedness and unimodal symmetric noise assumptions, which are stronger assumptions than our SPB assumption (see Lemma~\ref{lemma:main-rho}). Finally, Theorem 1 in \citep{BZAA} can be recovered from Theorem~\ref{non-convex-theorem} (see Appendix~\ref{sec:recovering}).

$\bullet$ \textbf{Convergence rate.} Rates (\ref{non-convex-rate-min}) and (\ref{const-step-rate}) can be arbitrarily slow, depending on the probabilities $\rho_i$. This is to be expected. At one extreme, if the gradient noise was completely random, i.e., if $\rho_i \equiv 1/2$, then the $\rho$-norm would become identical zero for any gradient vector and rates would be trivial inequalities, leading to divergence as in the counterexample. At other extreme, if there was no gradient noise, i.e., if $\rho_i\equiv 1$, then the $\rho$-norm would be just the $l^1$ norm and we get the rate $\cO(1/\sqrt{K})$ with respect to the $l^1$ norm. However, if  we know that $\rho_i>1/2$, then we can ensure that the method will eventually converge.

$\bullet$ \textbf{Geometry.} The presence of the $\rho$-norm in these rates suggests that, in general, there is no particular geometry (e.g., $l^1$ or $l^2$) associated with signSGD. Instead, the geometry is induced from the success probabilities. For example, in the case of unbiased and unimodal symmetric noise, the geometry is described by the mixture norm $l^{1,2}$.



Next, we state a general convergence rate for Algorithm \ref{alg:signSGD} with Option 2.

\begin{theorem}[see \ref{apx:non-convex-theorem-2}]\label{non-convex-theorem-2}
Under the SPB assumption, signSGD (Algorithm~\ref{alg:signSGD}) with Option 2 and with step sizes $\gamma_k = \gamma_0/\sqrt{k+1}$ converges as follows: 
\begin{equation*}
\frac{1}{K}\sum\limits_{k=0}^{K-1}\E\|\nabla f(x_k)\|_\rho \le \frac{1}{\sqrt{K}} \left[\frac{f(x_0)-f^*}{\gamma_0} + \gamma_0 d\bar{L}\right].
\end{equation*}
In the case of $\gamma_k \equiv \gamma>0$, the same rate as \eqref{const-step-rate} is achieved.
\end{theorem}

Comparing Theorem~\ref{non-convex-theorem-2} with Theorem~\ref{non-convex-theorem}, notice that one can remove the log factor from \eqref{non-convex-rate-min} and bound the average of past gradient norms instead of the minimum. On the other hand, in a big data regime, function evaluations in Algorithm~\ref{alg:signSGD} (Option~2, line~4) are infeasible. Clearly, Option~2 is useful {\em only} in the setup when one can afford function evaluations and has rough estimates about the gradients (i.e., signs of stochastic gradients). This option should be considered within the framework of derivative-free optimization.

\subsection{Convergence analysis in parallel setting} \label{sec:parallel}
In this part we present the convergence result of parallel signSGD (Algorithm \ref{alg:d-signSGD}) with majority vote introduced by \citet{BWAA}. Majority vote is considered within a parameter server framework, where for each coordinate parameter server receives one sign from each node and sends back the sign sent by the majority of nodes. In parallel setting, the training data is shared among the nodes.

\begin{algorithm}[H]
\begin{algorithmic}[1]
\STATE \textbf{Input:} step size $\gamma_k$, current point $x_k$, \# of nodes $M$
\STATE \textbf{on} each node $n$
\STATE \qquad $\hat{g}^n(x_k) \leftarrow \textrm{StochasticGradient}(f, x_k)$
\STATE \textbf{on} server
\STATE \qquad \textbf{get} $\sign\hat{g}^n(x_k)$ \textbf{from} all nodes
\STATE \qquad \textbf{send} $\sign\left[\sum_{n=1}^M \sign\hat{g}^n(x_k)\right]$ \textbf{to} all nodes
\STATE \textbf{on} each node $n$
\STATE \qquad $x_{k+1} = x_k - \gamma_k\sign\left[\sum_{n=1}^M \sign\hat{g}^n(x_k)\right]$
\end{algorithmic}
\caption{\sc Parallel signSGD with Majority Vote}
\label{alg:d-signSGD}
\end{algorithm}

Known convergence results \citep{BWAA, BZAA} use $\cO(K)$ mini-batch size as well as $\cO(1/K)$ constant step size. In the sequel we remove this limitations extending Theorem~\ref{non-convex-theorem} to parallel training. In this case the number of nodes $M$ get involved in geometry introducing new $\rho_M$-norm, which is defined by the regularized incomplete beta function $I$ (see Appendix~\ref{apx:non-convex-theorem-d}).

\begin{definition}[$\rho_M$-norm]
Let $M$ be the number of nodes and $l \eqdef \lfloor\tfrac{M+1}{2}\rfloor$. Define $\rho_M$-norm of gradient $g(x)$ at $x\in\R^d$ via
$$
\|g(x)\|_{\rho_M} \eqdef \sum_{i=1}^d \left(2 I(\rho_i(x); l, l) - 1\right)|g_i(x)|.
$$
\end{definition}
Clearly, $\rho_1$-norm coincides with $\rho$-norm.
Now we state the convergence rate of parallel signSGD with majority vote.
\begin{theorem}[see \ref{apx:non-convex-theorem-d}]\label{non-convex-theorem-d}
Under SPB assumption, parallel signSGD (Algorithm \ref{alg:d-signSGD}) with step sizes $\gamma_k = \gamma_0/\sqrt{k+1}$ converges as follows
\begin{equation}\label{non-convex-d-rate-min}
\min\limits_{0\le k<K} \E\|\nabla f(x_k)\|_{\rho_M}
\le
\frac{f(x_0) - f^*}{\gamma_0\sqrt{K}} + \frac{3\gamma_0 d\bar{L}}{2}\frac{\log{K}}{\sqrt{K}}.
\end{equation}
For constant step sizes $\gamma_k \equiv \gamma>0$, we have convergence up to a level proportional to step size $\gamma$:
\begin{equation}\label{const-step-d-rate}
\frac{1}{K}\sum_{k=0}^{K-1} \E\|\nabla f(x_k)\|_{\rho_M} \le \frac{f(x_0) - f^*}{\gamma K} + \frac{\gamma d\bar{L}}{2}.
\end{equation}
\end{theorem}

$\bullet$ \textbf{Speedup with respect to $\bm{M}$.}
Note that, in parallel setting with $M$ nodes, the only difference in convergence rates (\ref{non-convex-d-rate-min}) and (\ref{const-step-d-rate}) is the modified $\rho_M$-norm measuring the size of gradients. Using Hoeffding's inequality, we show (see Appendix~\ref{apx:exp-vr}) that $\|g(x)\|_{\rho_M} \to \|g(x)\|_1$ exponentially fast as $M\to\infty$, namely
$$
\left(1-e^{-(2\rho(x)-1)^2 l}\right)\|g(x)\|_1  \le  \|g(x)\|_{\rho_M}  \le  \|g(x)\|_1,
$$
where $\rho(x) = \min_{1\le i\le d}\rho_i(x) > \nicefrac{1}{2}$.
To appreciate the speedup with respect to $M$, consider the noise-free case as a baseline, for which $\rho_i\equiv 1$ and $\|g(x)\|_{\rho_M} \equiv \|g(x)\|_1$. Then, the above inequality implies that $M$ parallel machines reduce the variance of gradient noise exponentially fast.

$\bullet$ \textbf{Number of Nodes.} Theoretically there is no difference between $2l-1$ and $2l$ nodes, and this is not a limitation of the analysis. Indeed, as it is shown in the proof, expected sign vector at the master with $M=2l-1$ nodes is the same as with $M=2l$ nodes:
$$
\E \sign (\hat{g}^{(2l)}_i\cdot g_i) = \E \sign(\hat{g}^{(2l-1)}_i\cdot g_i),
$$
where $\hat{g}^{(M)}$ is the sum of stochastic sign vectors aggregated from nodes. Intuitively, majority vote with even number of nodes, e.g. $M=2l$,  fails to provide any sign with little probability (it is the probability of half nodes voting for $+1$, and half nodes voting for $-1$). However, if we remove one node, e.g. $M=2l-1$, then master receives one sign-vote less but gets rid of that little probability of failing the vote (sum of odd number of $\pm 1$ cannot vanish).

\subsection{Distributed sign methods and the homogeneous setting} \label{sec:distributed}
First, we discuss a fundamental limitation of signSGD for heterogeneous local objectives and then analyze our new sign based method in the homogeneous setting.

\subsubsection{The issue with heterogeneous distributed signSGD}
Consider distributed training where each machine $n\in\{1,2,\dots,M\}$ has its own loss function $f_n(x)$. We argue that in this setting even signGD (with full-batch gradients and no noise) can fail to converge.
Indeed, let us multiply each loss function $f_n(x)$ of $n$th node by an arbitrary positive scalars $w_n>0$. Then the landscape (in particular, stationary points) of the overall loss function
$$
f^w(x) \eqdef \tfrac{1}{M}\sum_{n=1}^M w_nf_n(x)
$$
can change arbitrarily while the iterates of signGD remain the same as the master server aggregates the same signs $\sign(w_n\nabla f_n(x))=\sign\nabla f_n(x)$ regardless of the scalars $w_n>0$.
Thus, distributed signGD is unable to sense the weights $w_n>0$ modifying total loss function $f^w$ and cannot guarantee approximate stationary point unless loss functions $f_n$ have some special structures.
This scaling invariance prevents a general convergence guarantee for arbitrary heterogeneous local objectives when only sign information is communicated. We therefore analyze the method below in the homogeneous setting, where every worker has an unbiased stochastic gradient oracle for the same objective.

\subsubsection{Novel Sign-based Method.}
We design a distributed sign-based method--{\em Stochastic Sign Descent with Momentum (SSDM)}--including two additional layers: {\em stochastic sign} and {\em momentum}. Each worker uses its own stochastic samples or stochastic-gradient oracle for the shared objective.

Motivated by SPB assumption, we introduce our new notion of {\em stochastic sign} to replace the usual deterministic $\sign$.

\begin{definition}[Stochastic Sign]\label{def:stochastic-sign}
We define the stochastic sign operator $\widetilde{\sign}:\R^d\to\R^d$ via
\begin{equation*}
\left(\widetilde{\sign}\,g\right)_i =
\begin{cases}
+1, & \text{ with probability } \frac{1}{2} + \frac{1}{2}\frac{g_i}{\|g\|} \\
-1, & \text{ with probability } \frac{1}{2} - \frac{1}{2}\frac{g_i}{\|g\|}
\end{cases}
\end{equation*}
for $1 \le i \le d$ and $\widetilde{\sign}\,\mathbf{0} = \mathbf{0}$ with probability $1$.
\end{definition}

Technical importance of stochastic $\widetilde{\sign}$ is twofold. First, it satisfies the SPB assumption automatically, that is
$$
\Pr( (\widetilde{\sign}\,g)_i = \sign g_i) = \frac{1}{2} + \frac{1}{2}\frac{|g_i|}{\|g\|} > \frac{1}{2},
$$
if $g_i\ne 0$. Second, unlike the deterministic $\sign$ operator, it is unbiased with scaling factor $\|g\|$, namely $\E[\|g\|\;\widetilde{\sign}\,g] = g$. We describe our SSDM method formally in Algorithm \ref{alg:ssdm}.

\begin{algorithm}[H]
\begin{algorithmic}[1]
\STATE \textbf{Input:} step size parameter $\gamma$, momentum parameter $\beta$, \# of nodes $M$
\STATE \textbf{Initialize:} $x_0\in\R^d,\; m_{-1}^n = \hat{g}_0^n$ for all $n\in\{1, 2, \dots, M\}$
\FOR{$k = 0, 1, \dots, K-1$}
\STATE \qquad \textbf{on} each node $n$
\STATE \qquad\qquad $\hat{g}^n_k \leftarrow \textrm{StochasticGradient}(f, x_k)$ \hfill { \scriptsize Independent stochastic sample of $f$}
\STATE \qquad\qquad $m^n_k = \beta m_{k-1}^n + (1-\beta)\hat{g}_k^n$ \hfill { \scriptsize Update the momentum}
\STATE \qquad\qquad \textbf{send} $s^n_k \eqdef \widetilde{\sign}\;m_k^n$ {\bf to} the server \hfill { \scriptsize Communicate to the server}
\STATE \qquad\textbf{on} server
\STATE \qquad\qquad \textbf{send} $s_k \eqdef \sum_{n=1}^M s_k^n$ {\bf to} all nodes \hfill { \scriptsize Communicate to all nodes}
\STATE \qquad\textbf{on} each node $n$
\STATE \qquad\qquad $x_{k+1} = x_k - \frac{\gamma}{M}s_k$ \hfill { \scriptsize Main step: Update the global model}
\ENDFOR
\end{algorithmic}
\caption{\sc Stochastic Sign Descent with Momentum (SSDM)}
\label{alg:ssdm}
\end{algorithm}

For the convergence analysis below, we consider the homogeneous setting. The workers may use independent stochastic samples or separate stochastic oracles. We model stochastic gradient oracles using the standard bounded variance condition defined below:

\begin{assumption}[Bounded Variance]\label{asum-BV}
For any $x\in\R^d$, each node $n$ has access to an unbiased estimator $\hat{g}^n(x)$ with bounded variance $(\sigma^n)^2\ge0$, namely
$$
\E\left[\hat{g}^n(x)\right] = \nabla f(x), \quad \E\left[\|\hat{g}^n(x) - \nabla f(x)\|^2\right] \le (\sigma^n)^2.
$$
\end{assumption}

Now, we present our convergence result for SSDM method.

\begin{theorem}[see \ref{apx:ssdm}]\label{thm:ssdm}
We assume that $f$ is $L$-smooth and lower bounded. Under Assumption \ref{asum-BV}, $K\ge1$ iterations of SSDM (Algorithm \ref{alg:ssdm}) with momentum parameter $\beta=1-\frac{1}{\sqrt{K}}$ and step-size $\gamma=\frac{1}{K^{\nicefrac{3}{4}}}$ guarantee
\begin{equation*}
\frac{1}{K}\sum_{k=0}^{K-1}\E\|\nabla f(x^k)\|
\le \frac{1}{K^{\nicefrac{1}{4}}}\left[ 3\delta_f + 16\tilde{\sigma} + 8L\sqrt{d} + \frac{3Ld}{\sqrt{K}} \right],
\end{equation*}
where $\delta_f = f(x_0)-f^*, \tilde{\sigma} = \frac{1}{M}\sum\limits_{n=1}^M \sigma^n$.
\end{theorem}

$\bullet$ \textbf{Optimal rate using sign bits only.} Note that, for non-convex distributed training, SSDM has the same optimal asymptotic rate $\cO(\varepsilon^{-4})$ as SGD.
In contrast, existing analyses of signSGD and its momentum version Signum \cite{BWAA,BZAA} require increasingly larger mini-batches over the course of training. The shared-data guarantee above is distinct from optimization with fully heterogeneous or partitioned local objectives, for which the scaling argument in Section~\ref{sec:distributed} prevents a general sign-only convergence guarantee.
A general approach to handle biased compression operators, satisfying certain contraction property, is the {\em error feedback (EF)} mechanism proposed by \citet{SFDLY}.
In particular, EF-signSGD method of \citet{SQSM} fixes the convergence issues of signSGD in single node setup, overcoming SBP assumption.
Furthermore, for distributed training, \citet{DoubleSqueeze2019} applied the error feedback trick both for the server and nodes in their \mbox{DoubleSqueeze} method maintaining the same asymptotic rate with bi-directional compression.
However, in these methods, the contraction property of compression operator used by error feedback forces to communicate the magnitudes of local stochastic gradients together with the signs. This is not the case for sign-based methods considered in this work, where only sign bits are communicated between nodes and server.

$\bullet$ {\bf Noisy signSGD.}
In some sense, stochastic sign operator (see Definition \ref{def:stochastic-sign}) can be viewed as noisy version of standard deterministic $\sign$ operator and, similarly, our SSDM method can be viewed as noisy variant of signSGD with momentum. This observation reveals a connection to the noisy signSGD method of \citet{noisySignSGD}. Despite some similarities between the two methods, there are several technical aspects that SSDM excels their noisy signSGD.
First, the noise they add is {\em artificial} and requires a special care: too much noise blows the convergence, too little noise is unable to shrink the gap between median and mean. Moreover, as it is discussed in their paper, the variance of the noise must depend on $K$ (total number of iterations) and tend to $\infty$ with $K$ to guarantee convergence to stationary points in the limit. Meanwhile, the noise of SSDM is {\em natural} and does not need to be adjusted.
Next, the convergence bound (17) of \citep{noisySignSGD} is harder to interpret than the bound in our Theorem \ref{thm:ssdm} involving $l^2$ norms of the gradients {\em only}.
Besides, the convergence rate with respect to squared $l^2$ norm is $\cO(\nicefrac{d^{3/4}}{K^{1/4}})$, while the rate of SSDM with respect to squared $l^2$ norm is $\cO(\nicefrac{d}{\sqrt{K}})$, which is $\cO(\nicefrac{K^{1/4}}{d^{1/4}})$ times {\em faster}.
Lastly, it is explicitly written before Theorem 5 that the analysis assumes {\em full} gradient computation for all nodes. In contrast, SSDM is analyzed under a more general stochastic gradient oracle. 

{\bf $\bullet$ All-reduce compatible.} In contrast to signSGD with majority vote aggregation, SSDM supports partial aggregation of compressed stochastic signs $s_k^n$. In other words, compressed signs $s_k^n$ can be directly summed without additional decompression-compression steps. This allows SSDM to be implemented with efficient {\em all-reduce} operation instead of slower {\em all-gather} operation. Besides SSDM, only a few compression schemes in the literature satisfy this property and can be implemented with {\em all-reduce} operation, e.g., SGD with random sparsification \cite{sparseSGD}, GradiVeQ \cite{GradiVeQ}, PowerSGD \cite{PowerSGD}.

Finally, we show that the improved convergence theory and low communication cost of SSDM is due to the use of {\em both} stochastic sign operator and momentum term.

{\bf $\bullet$  SSDM without stochastic sign.} If we replace stochastic sign by deterministic sign in SSDM, then the resulting method {\em can provably diverge} even when full gradients are computed by all nodes. In fact, the counterexample \eqref{counter-ex-1node} in Section~\ref{sec:counterexample} can be easily extended to distributed setting and can handle momentum. Indeed, consider $M=2$ nodes owning functions $f_n(x) = \<a_n,x\>^2,\; n=1,2$ with $a_1,a_2$ as defined in \eqref{counter-ex-1node} and initial point $x_0\in Z=\{(z_1,z_2)\colon z_1+z_2=2\}$. Since $\nabla f_n(x) = 2\<a_n,x\>a_n \in span(a_n)$, we imply $m_k^n\in span(a_n)$ for any value of parameter $\beta$ and for all iterate $k\ge0$ (see lines 2 and 6 of Algorithm \ref{alg:ssdm}). Hence, $\sign m_k^n = \pm\sign a_n = \pm\left[\begin{smallmatrix} -1 \\ 1\end{smallmatrix}\right]$. Since $s_k = \sign m_k^1 + \sign m_k^2 \in span(\left[\begin{smallmatrix} -1 \\ 1\end{smallmatrix}\right])$ (see line 9), this means that the method is again stuck along the line $Z$ as $\frac{\gamma}{M}s_k \in span(\left[\begin{smallmatrix} -1 \\ 1\end{smallmatrix}\right])$ (see line 11) for any value of $\gamma$.

{\bf $\bullet$  SSDM without momentum.} It is possible to obtain the same asymptotic convergence rate without the momentum term (i.e., $\beta=0$). In this case, if all nodes also send the norms $\|\hat{g}_k^n\|$ to the server then the method can be analyzed by a standard analysis of distributed SGD with an unbiased compression. However, the drawback of this approach is the {\em higher communication cost}. While the overhead of worker-to-server communication is negligible (one extra float), the reverse server-to-worker communication becomes costly as the aggregated updates are dense (all entries are floats) as opposed to the original SSDM (all entries are integers).

\begin{figure}[ht]
\begin{center}
\centerline{\includegraphics[width=0.75\columnwidth]{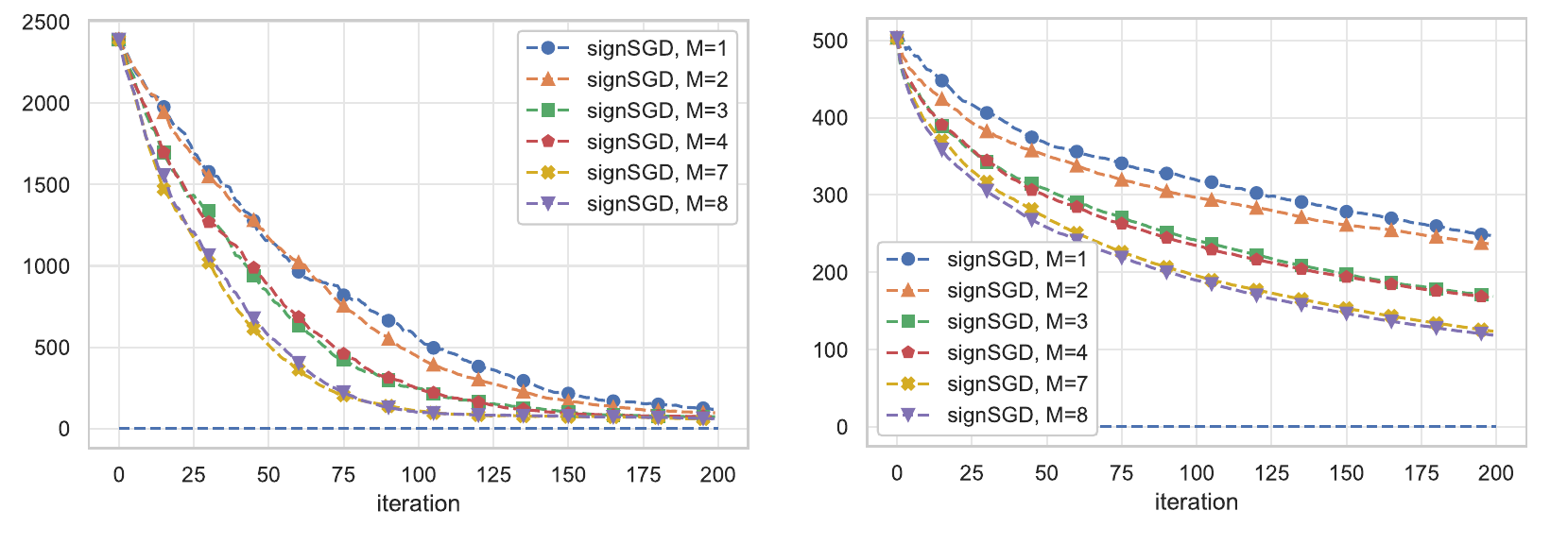}}
\caption{Experiments on distributed signSGD with majority vote using Rosenbrock function. Plots show function values with respect to iterations averaged over 10 repetitions. Left plot used constant step size $\gamma=0.02$, right plot used variable step size with $\gamma_0=0.02$. We set mini-batch size 1 and used the same initial point. Dashed blue lines mark the minimum.}
\label{fig:d-exp}
\end{center}
\vskip -0.4in
\end{figure}

\section{Experiments}\label{sec:exp}

We verify several aspects of our theoretical results experimentally using the MNIST dataset with feed-forward neural network (FNN) and the well known Rosenbrock (non-convex) function with $d=10$ variables:
$$
f(x) = \sum_{i=1}^{d-1} f_i(x), \quad\text{where}\quad f_i(x) = 100(x_{i+1} - x_i^2)^2 + (1-x_i)^2.
$$

\subsection{Minimizing the Rosenbrock function}

The Rosenbrock function is a classic example of non-convex function, which is used to test the performance of optimization methods. We chose this low dimensional function in order to estimate the success probabilities effectively in a reasonable time and to expose theoretical connection.

Stochastic formulation of the minimization problem for Rosenbrock function is as follows: at any point $x\in\R^d$ we have access to {\em biased} stochastic gradient $\hat{g}(x) = \nabla f_i(x) + \xi$, where index $i$ is chosen uniformly at random from $\{1, 2, \dots, d-1\}$ and $\xi\sim\mathcal{N}(0, \nu^2 I)$ with $\nu>0$.

Figure~\ref{fig:d-exp} illustrates the effect of multiple nodes in distributed training with majority vote. As we see increasing the number of nodes improves the convergence rate. It also supports the claim that in expectation there is no improvement from $2l-1$ nodes to $2l$ nodes. 

Figure \ref{fig:const-lr} shows the robustness of SPB assumption in the convergence rate (\ref{const-step-rate}) with constant step size. We exploited four levels of noise in each column to demonstrate the correlation between success probabilities and convergence rate. In the first experiment (first column) SPB assumption is violated strongly and the corresponding rate shows divergence. In the second column, probabilities still violating SPB assumption are close to the threshold and the rate shows oscillations. Next columns express the improvement in rates when success probabilities are pushed to be close to 1.
More experiments on the Rosenbrock function are moved to Appendix \ref{apx:add-exp}.


\begin{figure*}[ht]
\begin{center}
\centerline{\includegraphics[width=0.9\columnwidth]{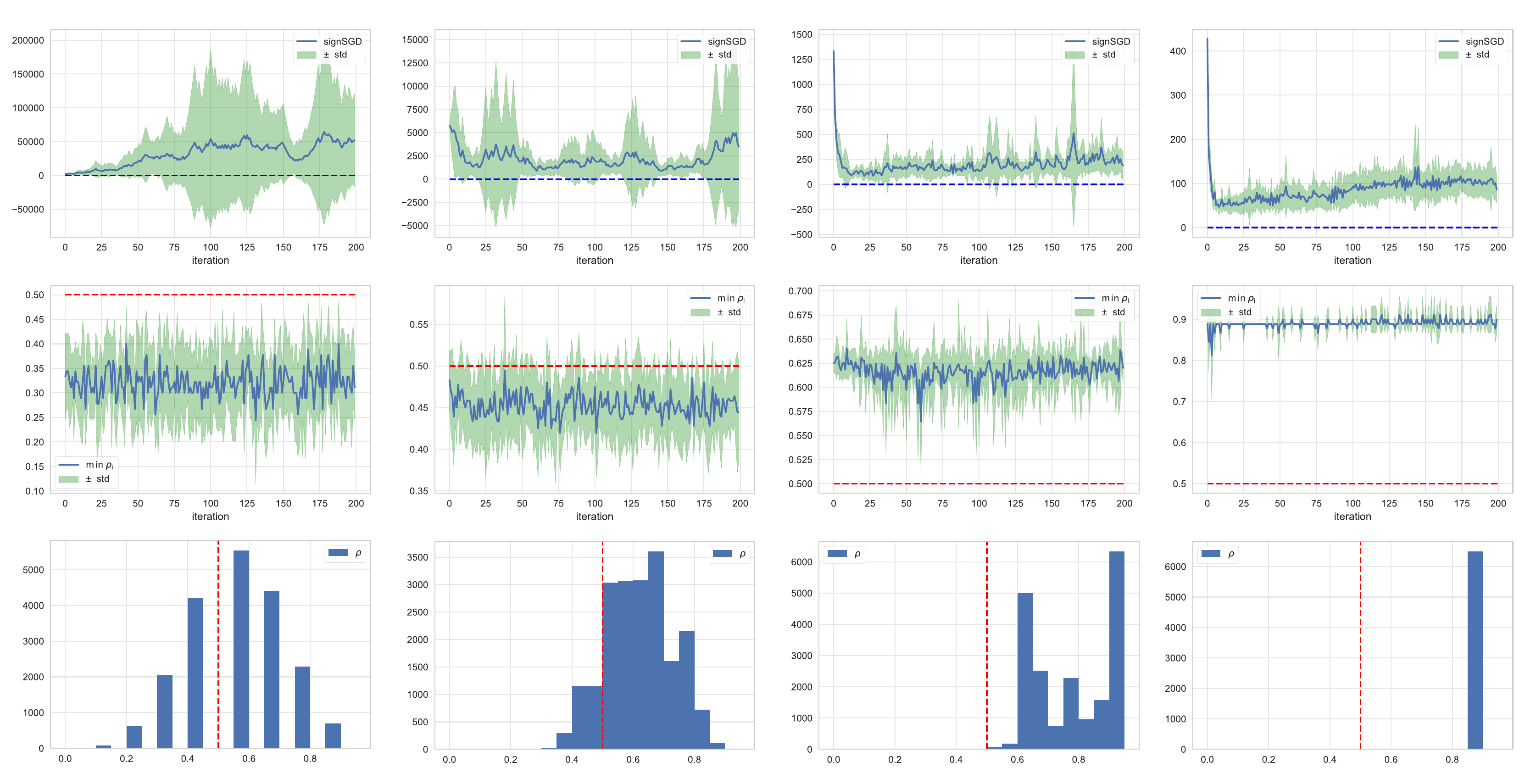}}
\caption{Performance of signSGD with constant step size ($\gamma=0.25$) under four different noise levels (mini-batch size 1, 2, 5, 8) using Rosenbrock function. Each column represent a separate experiment with function values, evolution of minimum success probabilities and the histogram of success probabilities throughout the iteration process. Dashed blue line in the first row is the minimum value. Dashed red lines in second and third rows are thresholds $1/2$ of success probabilities. The shaded area in first and second rows shows standard deviation obtained from ten repetitions.}
\label{fig:const-lr}
\end{center}
\vskip -0.4in
\end{figure*}


\subsection{Training FNN on the MNIST dataset}

We trained a single layer feed-forward network on the MNIST with two different batch construction strategies. The first construction is the standard way of training networks: before each epoch we shuffle the training dataset and choose batches sequentially. In the second construction, first we split the training dataset into two parts, images with labels 0, 1, 2, 3, 4 and images with labels 5, 6, 7, 8, 9. Then each batch of images were chosen from one of these parts with equal probabilities.
We make the following observations based on our experiments depicted in Figure \ref{figure:split-batch} and Figure \ref{figure:standard-batch}.

$\bullet$ \textbf{Convergence with multi-modal and skewed gradient distributions.} Due to the split batch construction strategy we unfold multi-modal and asymmetric distributions for stochastic gradients in Figure \ref{figure:split-batch}. With this experiment we conclude that sign based methods can converge under various gradient distributions which is allowed from our theory.

$\bullet$ \textbf{Effectiveness in the early stage of training.} Both experiments show that in the beginning of the training, signSGD is more efficient than SGD when we compare accuracy against communication. This observation is supported by the theory as at the start of the training success probabilities are bigger (see Lemma~\ref{lemma:main-rho}) and lower bound for mini-batch size is smaller (see Lemma \ref{lemma:SPB-via-minibatch}).

$\bullet$ \textbf{Bigger batch size, better convergence.} Figure \ref{figure:standard-batch} shows that the training with larger batch size improves the convergence as backed by the theory (see Lemmas \ref{lemma:strict-rho} and \ref{lemma:SPB-via-minibatch}).

$\bullet$ \textbf{Generalization effect.} Another aspect of sign based methods which has been observed to be problematic, in contrast to SGD, is the generalization ability of the model (see also \citep{BaHe}, Section~6.2 Results). In the experiment with standard batch construction (see Figure~\ref{figure:standard-batch}) we notice that test accuracy is growing with training accuracy. However, in the other experiment with split batch construction (see Figure~\ref{figure:split-batch}), we found that test accuracy does not get improved during the second half of the training while train accuracy grows consistently with slow pace.

\begin{figure*}[ht]
\begin{center}
\centerline{\includegraphics[width=\columnwidth]{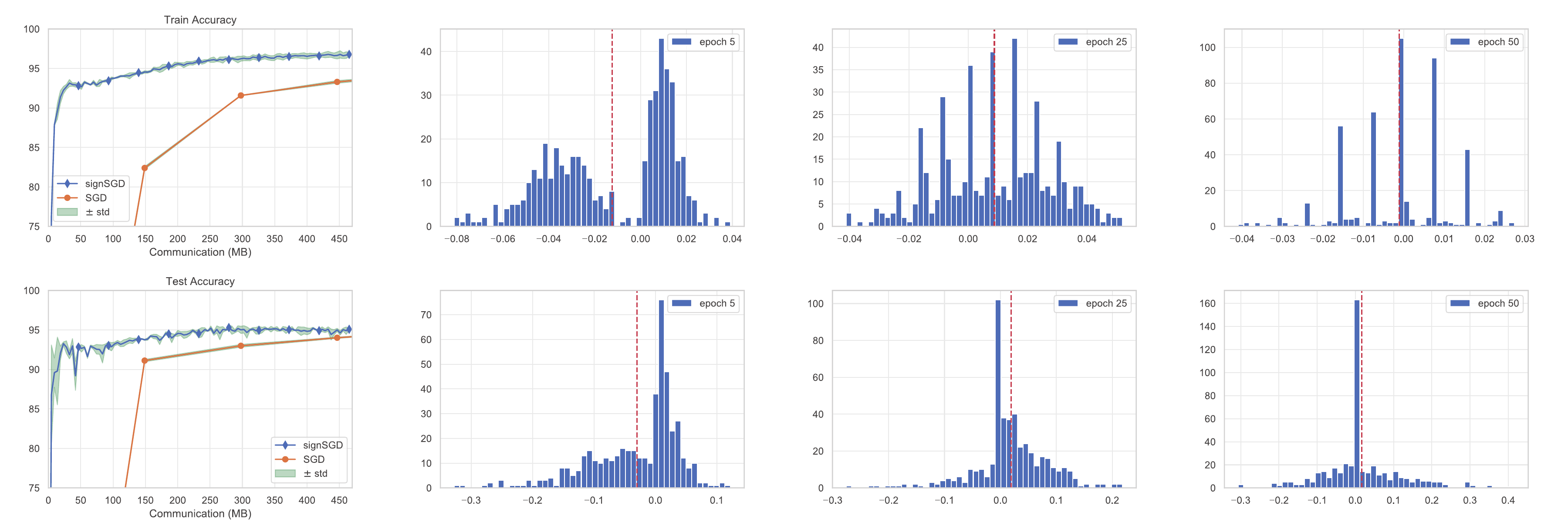}}
\caption{Convergence of signSGD and comparison with SGD on the MNIST dataset using the split batch construction strategy. The budget of gradient communication (MB) is fixed and the network is a single hidden layer FNN. We first tuned the constant step size over logarithmic scale $\{1, 0.1, 0.01, 0.001, 0.0001\}$ and then fine tuned it.
First column shows train and test accuracies with mini-batch size 128 and averaged over 3 repetitions. We chose two weights (empirically, most of the network biases would work) and plotted histograms of stochastic gradients before epochs 5, 25 and 50. Dashed red lines on histograms indicate the average values.}
\label{figure:split-batch}
\end{center}
\end{figure*}

\begin{figure*}[t]
\begin{center}
\centerline{\includegraphics[width=\columnwidth]{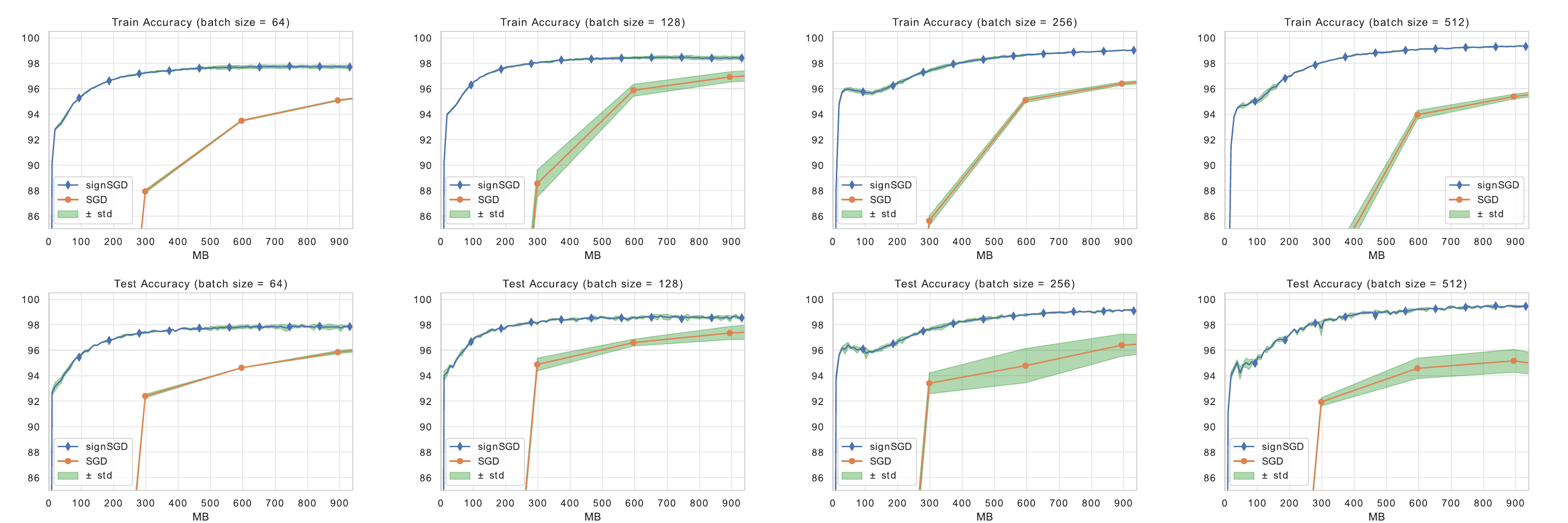}}
\caption{Comparison of signSGD and SGD on the MNIST dataset with a fixed budget of gradient communication (MB) using single hidden layer FNN and the standard batch construction strategy. For each batch size, we first tune the constant step size over logarithmic scale $\{10, 1, 0.1, 0.01, 0.001\}$ and then fine tune it. Shaded area shows the standard deviation from 3 repetition.}
\label{figure:standard-batch}
\end{center}
\end{figure*}

\clearpage
\section*{Acknowledgments}

Mher Safaryan thanks El Houcine Bergou, Konstantin Mishchenko, and Samuel Horváth for helpful discussions and feedback on the paper, and Anton Pliusnin for identifying the issue with Theorem~8 in the heterogeneous setting.

\bibliography{signsgd_references}

@conference{PowerSGD,
title = {{PowerSGD}: Practical Low-Rank Gradient Compression for Distributed Optimization},
author = {Thijs Vogels and Sai Praneeth Karimireddy and Martin Jaggi},
booktitle = {33th Advances in Neural Information Processing Systems},
year = {2019}
}

@conference{GradiVeQ,
title = {{GradiVeQ}: Vector Quantization for Bandwidth-Efficient Gradient Aggregation in Distributed {CNN} Training},
author = {Mingchao Yu and Zhifeng Lin and Krishna Narra and Songze Li and Youjie Li and Nam Sung Kim and Alexander Schwing and Murali Annavaram and Salman Avestimehr},
booktitle = {32th Advances in Neural Information Processing Systems},
year = {2018}
}

@conference{sparseSGD,
title = {Gradient sparsification for communication-efficient distributed optimization},
author = {Jianqiao Wangni and Jialei Wang and Ji Liu and Tong Zhang},
booktitle = {32th Advances in Neural Information Processing Systems},
year = {2018}
}

@conference{noisySignSGD,
title = {Distributed Training with Heterogeneous Data: Bridging Median- and Mean-Based Algorithms},
author = {Xiangyi Chen and Tiancong Chen and Haoran Sun and Zhiwei Steven Wu and Mingyi Hong},
booktitle = {34th Conference on Neural Information Processing Systems},
year = {2020}
}

@conference{DoubleSqueeze2019,
  title =    {$\texttt{DoubleSqueeze}$: Parallel Stochastic Gradient Descent with Double-pass Error-Compensated Compression},
  author =   {Tang, Hanlin and Yu, Chen and Lian, Xiangru and Zhang, Tong and Liu, Ji},
  booktitle =    {Int. Conf. Machine Learning},
  pages =    {6155--6165},
  year =   {2019},
  volume =   {PMLR 97},
  }

@InProceedings{N-SGD-M,
title = {Momentum Improves Normalized {SGD}},
author = {Cutkosky, Ashok and Mehta, Harsh},
booktitle = {Proceedings of the 37th International Conference on Machine Learning},
pages = {2260--2268},
year = {2020},
editor = {Hal Daumé III and Aarti Singh},
volume = {119},
series = {Proceedings of Machine Learning Research},
address = {Virtual},
month = {13--18 Jul},
publisher = {PMLR},
url = {http://proceedings.mlr.press/v119/cutkosky20b.html}
}

@inproceedings{AGLTV,
title = {{QSGD}: Communication-Efficient {SGD} via Gradient Quantization and Encoding},
author = {Alistarh, Dan and Grubic, Demjan and Li, Jerry and Tomioka, Ryota and Vojnovic, Milan},
booktitle = {Advances in Neural Information Processing Systems 30},
pages = {1709-1720},
year = {2017}
}

@inproceedings{BaHe,
title = {Dissecting {Adam}: The Sign, Magnitude and Variance of Stochastic Gradients},
author = {Balles, Lukas and Hennig, Philipp},
booktitle = {Proceedings of the 35th International Conference on Machine Learning},
pages = {404-413},
year = {2018}
}

@inproceedings{BWAA,
title = {sign{SGD}: Compressed Optimisation for Non-Convex Problems},
author = {Bernstein, Jeremy and Wang, Yu-Xiang and Azizzadenesheli, Kamyar and Anandkumar, Animashree},
booktitle = {Proceedings of the 35th International Conference on Machine Learning},
pages = {560-569},
volume =  {80},
publisher = {PMLR},
year = {2018}
}

@inproceedings{BZAA,
title = {sign{SGD} with majority vote is communication efficient and fault tolerant},
author = {Bernstein, Jeremy and Zhao, Jiawei and Azizzadenesheli, Kamyar and Anandkumar, Animashree},
booktitle = {International Conference on Learning Representations},
pages = {},
year = {2019}
}

@conference{BoLe,
title = {Large scale online learning},
author = {Bottou, L\'eon and Le Cun, Yann},
booktitle = {Advances in Neural Information Processing Systems},
year = {2003}
}

@conference{CCC,
title = {Stochastic Spectral Descent for Restricted Boltzmann Machines},
author = {Carlson, David and Cevher, Volkan and Carin, Lawrence},
booktitle = {International Conference on Artificial Intelligence and Statistics (AISTATS)},
pages = {111-119},
year = {2015}
}

@conference{DHS,
title = {Adaptive subgradient methods for online learning and stochastic optimization},
author = {Duchi, John and Hazan, Elad and Singer, Yoram},
booktitle = {Journal of Machine Learning Research},
pages = {2121–2159},
year = {2011}
}

@conference{GhLa,
title = {Stochastic first-and zeroth-order methods for nonconvex stochastic programming},
author = {Ghadimi, Saeed and Lan, Guanghui},
booktitle = {SIAM Journal on Optimization},
pages = {2341–2368},
volume = {23(4)},
year = {2013}
}

@conference{KFJ,
title = {Distributed learning with compressed gradients},
author = {Khirirat, Sarit and Feyzmahdavian, Hamid Reza and Johansson, Mikael},
booktitle = {arXiv preprint arXiv:1806.06573},
pages = {},
year = {2018}
}

@conference{KiBa,
title = {Adam: A method for stochastic optimization},
author = {Kingma, Diederik and Ba, Jimmy},
booktitle = {International Conference on Learning Representations},
pages = {},
year = {2015}
}

@conference{KSH,
title = {Imagenet classification with deep convolutional neural networks},
author = {Krizhevsky, Alex and Sutskever, Ilya and Hinton, Geoffrey E.},
booktitle = {Advances in Neural Information Processing Systems},
pages = {1097–1105},
year = {2012}
}

@conference{LHMWD,
title = {Deep gradient compression: Reducing the communication bandwidth for distributed training},
author = {Lin, Yujun and Han, Song and Mao, Huizi and Wang, Yu and Dally, William J.},
booktitle = {International Conference on Learning Representations},
pages = {},
year = {2018}
}

@conference{LCCH,
title = {sign{SGD} via zeroth-order oracle},
author = {Liu, Sijia and Chen, Pin-Yu and Chen, Xiangyi and Hong, Mingyi},
booktitle = {International Conference on Learning Representations},
pages = {},
year = {2019}
}

@InProceedings{SQSM,
  title = 	 {Error Feedback Fixes {S}ign{SGD} and other Gradient Compression Schemes},
  author = 	 {Karimireddy, Sai Praneeth and Rebjock, Quentin and Stich, Sebastian and Jaggi, Martin},
  booktitle = 	 {Proceedings of the 36th International Conference on Machine Learning},
  pages = 	 {3252--3261},
  year = 	 {2019},
  volume = 	 {97}
}

@conference{MGTR,
title = {Distributed Learning with Compressed Gradient Differences},
author = {Mishchenko, Konstantin and Gorbunov, Eduard and Tak\'a\v{c}, Martin and Richt\'arik, Peter},
booktitle = {arXiv preprint arXiv:1901.09269},
pages = {},
year = {2019}
}

@conference{QRGSLS,
title = {{SGD} with Arbitrary Sampling: General Analysis and Improved Rates},
author = {Qian, Xun and Richt\'arik, Peter and Gower, Robert Mansel and Sailanbayev, Alibek and Loizou, Nicolas and Shulgin, Egor},
booktitle = {International Conference on Machine Learning},
pages = {},
year = {2019}
}

@conference{RKK,
title = {On the convergence of {Adam} and beyond},
author = {Reddi, Sashank and Kale, Satyen and Kumar, Sanjiv},
booktitle = {International Conference on Learning Representations},
pages = {},
year = {2019}
}

@conference{RiBr,
title = {A direct adaptive method for faster backpropagation learning: The {Rprop} algorithm},
author = {Riedmiller, Martin and Braun, Heinrich},
booktitle = {IEEE International Conference on Neural Networks},
pages = {586-591},
year = {1993}
}

@conference{RoMo,
title = {A Stochastic Approximation Method},
author = {Robbins, Herbert and Monro, Sutton},
booktitle = {The Annals of Mathematical Statistics},
pages = {400-407},
volume = {22(3)},
year = {1951}
}

@conference{Sch,
title = {Deep learning in neural networks: An overview},
author = {Schmidhuber, J\"urgen},
booktitle = {Neural networks},
pages = {85–117},
volume = {61},
year = {2015}
}

@conference{SR,
title = {Fast convergence of stochastic gradient descent under a strong growth condition},
author = {Schmidt, Mark and Le Roux, Nicolas},
booktitle = {arXiv preprint arXiv:1308.6370},
pages = {},
year = {2013}
}

@conference{SFDLY,
title = {1-Bit Stochastic Gradient Descent and Application to Data-Parallel Distributed Training of Speech {DNN}s},
author = {Seide, Frank and Fu, Hao and Droppo, Jasha and Li, Gang and Yu, Dong},
booktitle = {Fifteenth Annual Conference of the International Speech Communication Association},
pages = {},
year = {2014}
}

@conference{Shev,
title = {On the absolute constants in the Berry–Esseen type inequalities for identically distributed summands},
author = {Shevtsova, Irina},
booktitle = {arXiv preprint arXiv:1111.6554},
pages = {},
year = {2011}
}

@conference{SRB,
title = {Minimizing finite sums with the stochastic average gradient},
author = {Schmidt, Mark and Roux, Nicolas Le and Bach, Francis},
booktitle = {Mathematical Programming},
pages = {83–112},
volume = {162(1-2)},
year = {2017}
}

@conference{Strom,
title = {Scalable distributed {DNN} training using commodity {GPU} cloud computing},
author = {Strom, Nikko},
booktitle = {Sixteenth Annual Conference of the International Speech Communication Association},
pages = {},
year = {2015}
}

@conference{TiHi,
title = {{RMSprop}},
author = {Tieleman, Tijmen and Hinton, Geoffrey E.},
booktitle = {Coursera: Neural Networks for Machine Learning, Lecture 6.5},
pages = {},
year = {2012}
}

@book{Versh,
place={Cambridge},
series={Cambridge Series in Statistical and Probabilistic Mathematics},
title={High-Dimensional Probability: An Introduction with Applications in Data Science},
DOI={10.1017/9781108231596},
publisher={Cambridge University Press},
author={Vershynin, Roman},
year={2018},
collection={Cambridge Series in Statistical and Probabilistic Mathematics}}

@conference{VBS,
title = {Fast and Faster Convergence of {SGD} for Over-Parameterized Models (and an Accelerated Perceptron)},
author = {Vaswani, Sharan and Bach, Francis and Schmidt, Mark},
booktitle = {Proceedings of the 22nd International Conference on Artificial Intelligence and Statistics, PMLR},
pages = {},
volume = {89},
year = {2019}
}

@conference{WSLCPW,
title = {Atomo: Communication-efficient learning via atomic sparsification},
author = {Wang, Hongyi and Sievert, Scott and Liu, Shengchao and Charles, Zachary and Papailiopoulos, Dimitris and Wright, Stephen},
booktitle = {Advances in Neural Information Processing Systems},
pages = {},
year = {2018}
}

@conference{WXYWWCL,
title = {Terngrad: Ternary gradients to reduce communication in distributed deep learning},
author = {Wen, Wei and Xu, Cong and Yan, Feng and Wu, Chunpeng and Wang, Yandan and Chen, Yiran and Li, Hai},
booktitle = {Advances in Neural Information Processing Systems},
pages = {1509–1519},
year = {2017}
}

@conference{WRSSR,
title = {The marginal value of adaptive gradient methods in machine learning},
author = {Wilson, Ashia and Roelofs, Rebecca and Stern, Mitchell and Srebro, Nati and Recht, Benjamin},
booktitle = {Advances in Neural Information Processing Systems},
pages = {4148-4158},
year = {2017}
}

@conference{ZRSKK,
title = {Adaptive methods for nonconvex optimization},
author = {Zaheer, Manzil and Reddi, Sashank and Sachan, Devendra and Kale, Satyen and Kumar, Sanjiv},
booktitle = {Advances in Neural Information Processing Systems},
pages = {9815-9825},
year = {2018}
}

@conference{Zei,
author = {{Zeiler}, Matthew D.},
title = "{ADADELTA: An Adaptive Learning Rate Method}",
booktitle = {arXiv e-prints, arXiv:1212.5701},
year = "2012"
}

@conference{ZLKALZ,
title = {{ZipML}: Training linear models with end-to-end low precision, and a little bit of deep learning},
author = {Zhang, Hantian and Li, Jerry and Kara, Kaan and Alistarh, Dan and Liu, Ji and Zhang, Ce},
booktitle = {Proceedings of the 34th International Conference on Machine Learning},
pages = {4035–4043},
volume = {70},
year = {2017}
}
\bibliographystyle{plainnat}

\clearpage
\tableofcontents

\clearpage
\appendix
\part*{Appendix}


\section{Additional Experiments}\label{apx:add-exp}

In this section we present several more experiments on the Rosenbrock function for further insights.

Figure \ref{fig:var-lr} experiments with the same setup but variable learning rate.
In Figure \ref{fig:neigh-size}, we investigated the size of the neighborhood with respect to step size.


\begin{figure}[ht]
\vskip 0.2in
\begin{center}
\centerline{\includegraphics[width=0.95\columnwidth]{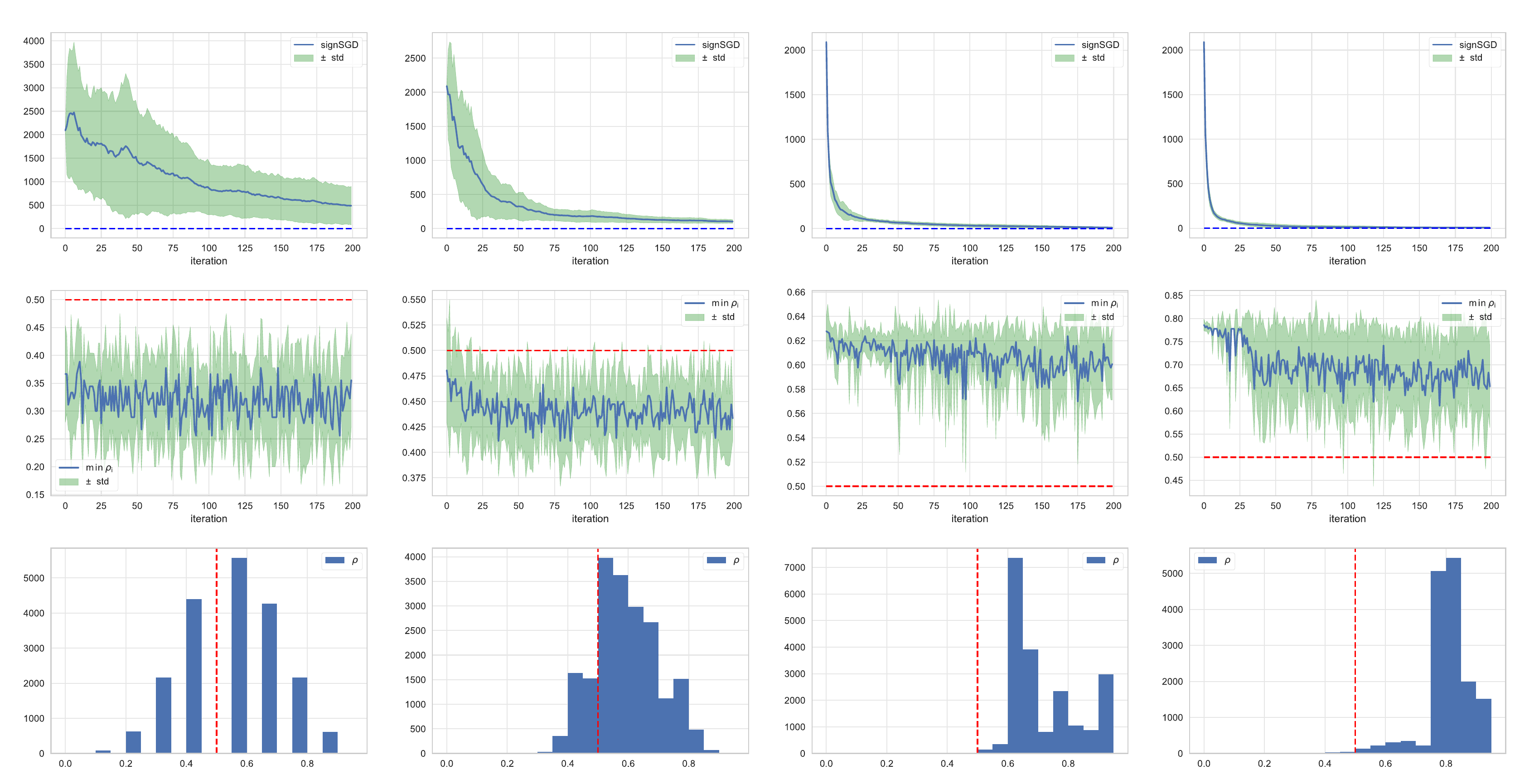}}
\caption{Performance of signSGD with variable step size ($\gamma_0=0.25$) under four different noise levels (mini-batch size 1, 2, 5, 7) using Rosenbrock function. As in the experiments of Figure \ref{fig:const-lr} with constant step size, these plots show the relationship between success probabilities and the convergence rate (\ref{non-convex-rate-min}). In low success probability regime (first and second columns) we observe oscillations, while in high success probability regime (third and forth columns) oscillations are mitigated substantially.}
\label{fig:var-lr}
\end{center}
\vskip -0.2in
\end{figure}

\begin{figure}[ht]
\vskip 0.2in
\begin{center}
\centerline{\includegraphics[width=0.95\columnwidth]{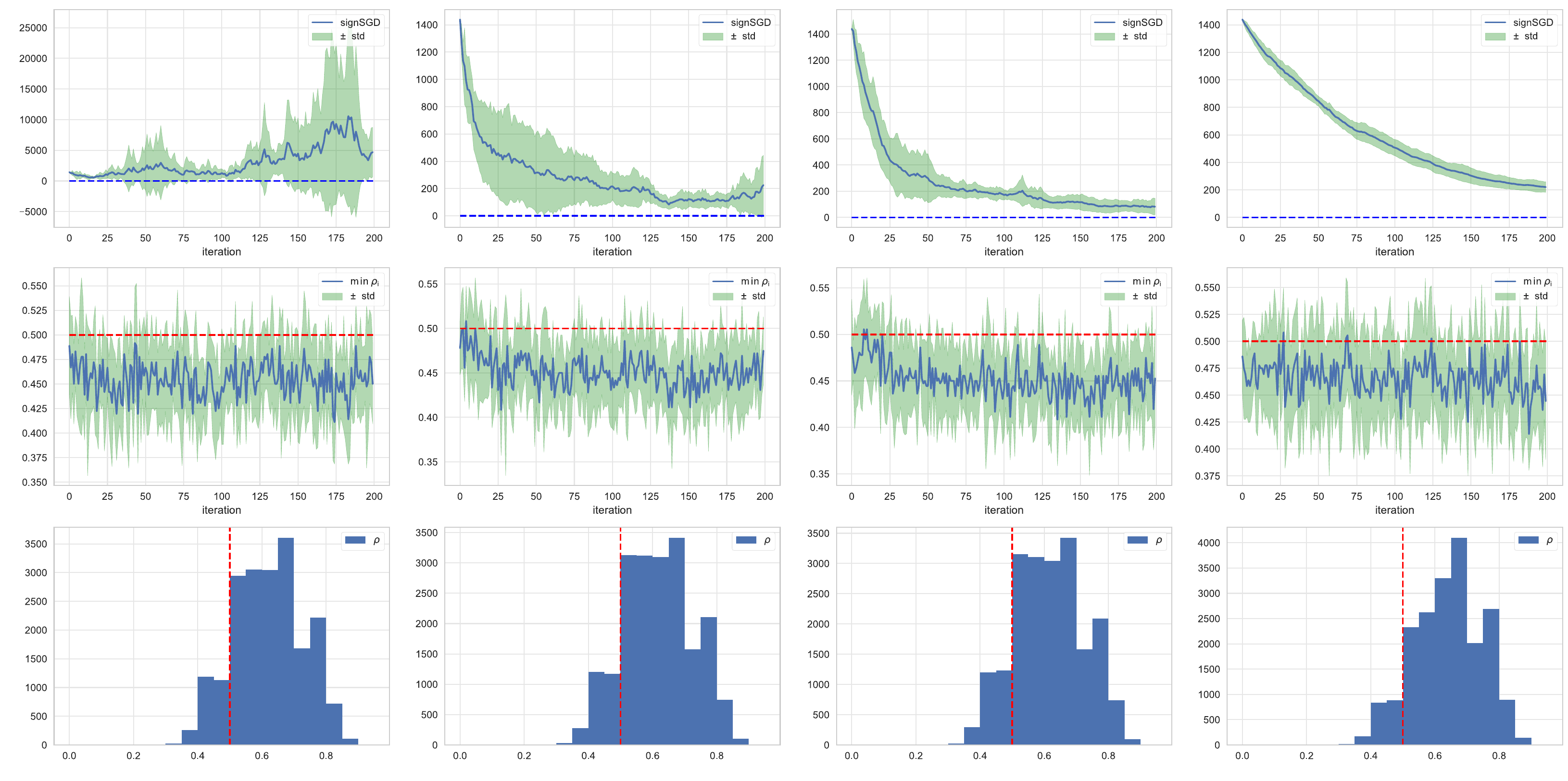}}
\caption{In this part of experiments we investigated convergence rate (\ref{const-step-rate}) to a neighborhood of the solution. We fixed gradient noise level by setting mini-batch size 2 and altered the constant step size. For the first column we set bigger step size $\gamma=0.25$ to detect the divergence (as we slightly violated SPB assumption). Then for the $2^\textrm{nd}$ and $3^\textrm{rd}$ columns we set $\gamma=0.1$ and $\gamma=0.05$ to expose the convergence to a neighborhood. For the forth column we set even smaller step size $\gamma=0.01$ to observe a slower convergence.}
\label{fig:neigh-size}
\end{center}
\vskip -0.2in
\end{figure}

\begin{figure}[ht]
\vskip 0.2in
\begin{center}
\centerline{\includegraphics[width=0.5\columnwidth]{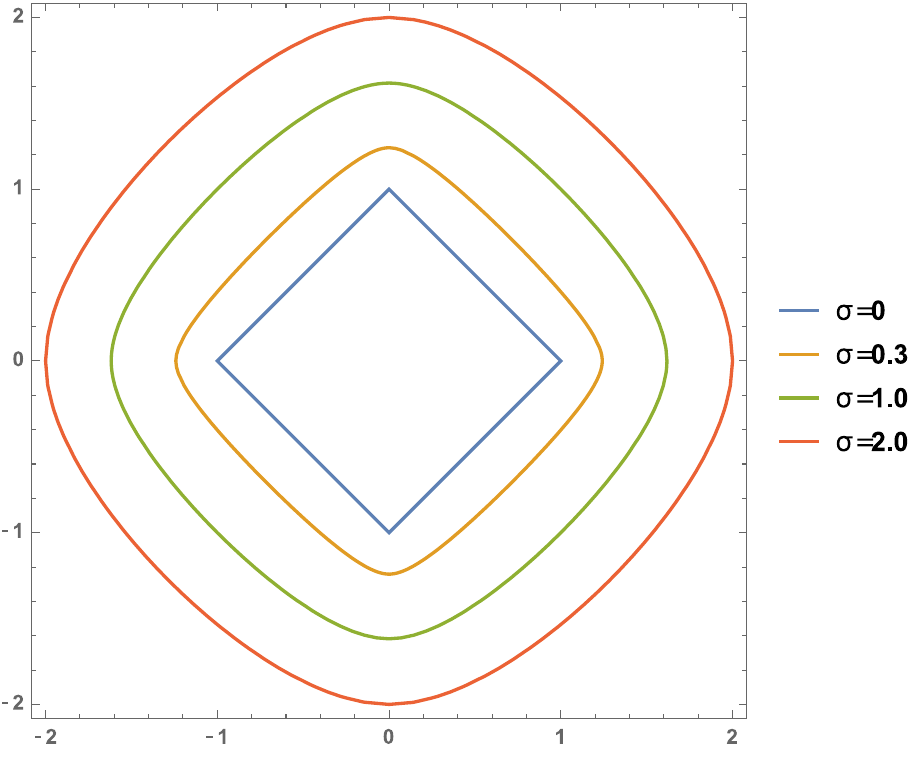}}
\caption{Unit balls in $l^{1,2}$ norm (\ref{def:l122-norm}) with different noise levels.}
\end{center}
\vskip -0.2in
\end{figure}

\clearpage

\section{Proofs}

\subsection{Sufficient conditions for SPB: Proof of Lemma \ref{lemma:main-rho}}\label{apx:lemma-main-rho}

Here we state the well-known Gauss's inequality on unimodal distributions\footnote{see \url{https://en.wikipedia.org/wiki/Gauss\%27s_inequality}}.

\begin{theorem}[Gauss's inequality]
Let $X$ be a unimodal random variable with mode $m$, and let $\sigma_m^2$ be the expected value of $(X-m)^2$. Then for any positive value of $r$,
\begin{equation*}
\Pr(|X-m| > r)
\le
\begin{cases}
\frac{4}{9}\left(\frac{\sigma_m}{r}\right)^2, & \text{ if } r\ge\frac{2}{\sqrt{3}}\sigma_m \\
1 - \frac{1}{\sqrt{3}}\frac{r}{\sigma_m}, & \text{ otherwise}
\end{cases}
\end{equation*}
\end{theorem}

Applying this inequality on unimodal and symmetric distributions, direct algebraic manipulations give the following bound:
\begin{equation*}
\Pr(|X-\mu| \le r)
\ge
\begin{cases}
1 - \frac{4}{9}\left(\frac{\sigma}{r}\right)^2, & \text{ if } \frac{\sigma}{r}\le\frac{\sqrt{3}}{2} \\
\frac{1}{\sqrt{3}}\frac{r}{\sigma}, & \text{ otherwise}
\end{cases}\,\,
\ge \frac{r/\sigma}{r/\sigma + \sqrt{3}},
\end{equation*}
where $m=\mu$ and $\sigma^2_m=\sigma^2$ are the mean and variance of unimodal, symmetric random variable $X$, and $r\ge 0$. Now, using the assumption that each $\hat{g}_i(x)$ has unimodal and symmetric distribution, we apply this bound for $X = \hat{g}_i(x),\,\mu = g_i(x), \sigma^2 = \sigma_i^2(x)$ and get a bound for success probabilities
\begin{align*}
\begin{split}
\Pr(\sign\hat{g}_i = \sign g_i)
& = 
\begin{cases}
\Pr(\hat{g}_i \ge 0), & \text{ if } g_i > 0 \\
\Pr(\hat{g}_i \le 0), & \text{ if } g_i < 0
\end{cases} \\
& = 
\begin{cases}
\frac{1}{2} + \Pr(0 \le \hat{g}_i \le g_i), & \text{ if } g_i > 0 \\
\frac{1}{2} + \Pr( g_i \le \hat{g}_i \le 0), & \text{ if } g_i < 0
\end{cases} \\
& = 
\begin{cases}
\frac{1}{2} + \frac{1}{2}\Pr(0 \le \hat{g}_i \le 2g_i), & \text{ if } g_i > 0 \\
\frac{1}{2} + \frac{1}{2}\Pr(2g_i \le \hat{g}_i \le 0), & \text{ if } g_i < 0
\end{cases} \\
& =  \frac{1}{2} + \frac{1}{2}\Pr(|\hat{g}_i - g_i| \le |g_i|) \\
&\ge \frac{1}{2} + \frac{1}{2}\frac{|g_i|/\sigma_i}{|g_i|/\sigma_i + \sqrt{3}} \\
& =  \frac{1}{2} + \frac{1}{2}\frac{|g_i|}{|g_i| + \sqrt{3}\sigma_i}
\end{split}
\end{align*}

{\bf Improvement on Lemma \ref{lemma:main-rho} and $l^{1,2}$ norm:} The bound after Gauss inequality can be improved including a second order term
\begin{equation*}
\Pr(|X-\mu| \le r)
\ge
\begin{cases}
1 - \frac{4}{9}\left(\frac{\sigma}{r}\right)^2, & \text{ if } \frac{\sigma}{r}\le\frac{\sqrt{3}}{2} \\
\frac{1}{\sqrt{3}}\frac{r}{\sigma}, & \text{ otherwise}
\end{cases}\,\,
\ge 1 - \frac{1}{1 + \nicefrac{r}{\sqrt{3}\sigma} + (\nicefrac{r}{\sqrt{3}\sigma})^2}.
\end{equation*}

Indeed, letting $z := \nicefrac{r}{\sqrt{3}\sigma} \ge \nicefrac{2}{3}$, we get $1-\frac{4}{9}\frac{1}{3z^2} \ge 1 - \frac{1}{1+z+z^2}$ as it reduces to $23z^2-4z-4\ge0$. Otherwise, if $0\le z\le\nicefrac{2}{3}$, then $z \ge 1 - \frac{1}{1+z+z^2}$ as it reduces to $1\ge 1-z^3$.
The improvement is tighter as
\begin{equation*}
\frac{r/\sigma}{r/\sigma + \sqrt{3}} = 1 - \frac{1}{1 + \nicefrac{r}{\sqrt{3}\sigma}} \le 1 - \frac{1}{1 + \nicefrac{r}{\sqrt{3}\sigma} + (\nicefrac{r}{\sqrt{3}\sigma})^2}.
\end{equation*}

Hence, continuing the proof of Lemma \ref{lemma:main-rho}, we get
\begin{equation*}
\Pr(\sign\hat{g}_i = \sign g_i) \ge  1 - \frac{1}{2}\frac{1}{1 + \nicefrac{|g_i|}{\sqrt{3}\sigma_i} + (\nicefrac{|g_i|}{\sqrt{3}\sigma_i})^2}
\end{equation*}

and we could have defined $l^{1,2}$-norm in a bit more complicated form as
\begin{equation*}
\|g\|_{l^{1,2}} := \sum_{i=1}^d \left(1- \frac{1}{1 + \nicefrac{|g_i|}{\sqrt{3}\sigma_i} + (\nicefrac{|g_i|}{\sqrt{3}\sigma_i})^2} \right)|g_i|.
\end{equation*}

\subsection{Sufficient conditions for SPB: Proof of Lemma \ref{lemma:strict-rho}}\label{apx:lemma-strict-rho}

Let $\hat{g}^{(\tau)}$ be the gradient estimator with mini-batch size $\tau$. It is known that the variance for $\hat{g}^{(\tau)}$ is dropped by at least a factor of $\tau$, i.e.
$$
\E[(\hat{g}_i^{(\tau)} - g_i)^2] \le \frac{\sigma_i^2}{\tau}.
$$

Hence, estimating the failure probabilities of $\sign\hat{g}^{(\tau)}$ when $g_i\neq0$, we have
\begin{align*}
\begin{split}
\Pr(\sign\hat{g}_i^{(\tau)} \neq \sign g_i)
& =   \Pr(|\hat{g}_i^{(\tau)} - g_i| = |\hat{g}_i^{(\tau)}| + |g_i|) \\
& \le \Pr(|\hat{g}_i^{(\tau)} - g_i| \ge |g_i|) \\
& =   \Pr((\hat{g}_i^{(\tau)} - g_i)^2 \ge g_i^2) \\
& \le \frac{\E[(\hat{g}_i^{(\tau)} - g_i)^2]}{g_i^2} \\
& =   \frac{\sigma_i^2}{\tau g_i^2},
\end{split}
\end{align*}
which imples
\begin{equation*}
\rho_i = \Pr(\sign\hat{g}_i = \sign g_i) \ge 1 - \frac{\sigma_i^2}{\tau g_i^2} \ge 1 - \frac{c_i}{\tau}.
\end{equation*}

\subsection{Sufficient conditions for SPB: Proof of Lemma \ref{lemma:SPB-via-minibatch}}\label{apx:lemma-SPB-via-minibatch}

The proof of this lemma is the most technical one. We will split the derivation into three lemmas providing some intuition on the way. The first two lemmas establish success probability bounds in terms of mini-batch size. Essentially, we present two methods: one works well in the case of small randomness, while the other one in the case of non-small randomness. In the third lemma, we combine those two bounds to get the condition on mini-batch size ensuring SPB assumption.

\begin{lemma}\label{lemma:SPB-small-entropy}
Let $X_1, X_2, \dots, X_{\tau}$ be i.i.d. random variables with non-zero mean $\mu := \E X_1 \ne 0$, finite variance $\sigma^2 := \E|X_1-\mu|^2 < \infty$. Then for any mini-batch size $\tau\ge1$
\begin{equation}\label{SPB-small-entropy}
\Pr\left(\sign\left[\frac{1}{\tau}\sum_{i=1}^{\tau}X_i\right] = \sign \mu \right) \ge 1 - \frac{\sigma^2}{\tau \mu^2}.
\end{equation}
\end{lemma}
\begin{proof}
Without loss of generality, we assume $\mu>0$. Then, after some adjustments, the proof follows from the Chebyshev's inequality:
\begin{align*}
\begin{split}
\Pr\left(\sign\left[\frac{1}{\tau}\sum_{i=1}^{\tau}X_i\right] = \sign \mu \right)
&=   \Pr\left(\frac{1}{\tau}\sum_{i=1}^{\tau}X_i > 0 \right) \\
&\ge \Pr\left(\left|\frac{1}{\tau}\sum_{i=1}^{\tau}X_i - \mu\right| < \mu \right) \\
&=   1 - \Pr\left(\left|\frac{1}{\tau}\sum_{i=1}^{\tau}X_i - \mu\right| \ge \mu \right) \\
&\ge 1 - \frac{1}{\mu^2} \var\left[\frac{1}{\tau}\sum_{i=1}^{\tau}X_i\right] \\
&=   1 - \frac{\sigma^2}{\tau \mu^2},
\end{split}
\end{align*}
where in the last step we used independence of random variables $X_1, X_2, \dots, X_{\tau}$.
\end{proof}

Obviously, bound (\ref{SPB-small-entropy}) is not optimal for big variance as it becomes a trivial inequality. In the case of non-small randomness a better bound is achievable additionally assuming the finiteness of 3th central moment.

\begin{lemma}\label{lemma:SPB-nonsmall-entropy}
Let $X_1, X_2, \dots, X_{\tau}$ be i.i.d. random variables with non-zero mean $\mu := \E X_1 \ne 0$, positive variance $\sigma^2 := \E|X_1-\mu|^2 > 0$ and finite 3th central moment $\nu^3 := \E|X_1-\mu|^3<\infty$. Then for any mini-batch size $\tau\ge1$
\begin{equation}\label{SPB-nonsmall-entropy}
\Pr\left(\sign\left[\frac{1}{\tau}\sum_{i=1}^{\tau}X_i\right] = \sign \mu \right) \ge \frac{1}{2}\left(1 + \erf\left(\frac{|\mu|\sqrt{\tau}}{\sqrt{2}\sigma}\right) - \frac{\nu^3}{\sigma^3\sqrt{\tau}}\right),
\end{equation}
where error function $\erf$ is defined as
\begin{equation*}
\erf(x) = \frac{2}{\sqrt{\pi}}\int_0^x e^{-t^2}\,dt, \quad x\in\R.
\end{equation*}
\end{lemma}
\begin{proof}
Again, without loss of generality, we may assume that $\mu>0$. Informally, the proof goes as follows. As we have an average of i.i.d. random variables, we approximate it (in the sense of distribution) by normal distribution using the Central Limit Theorem (CLT). Then we compute success probabilities for normal distribution with the error function $\erf$. Finally, we take into account the approximation error in CLT, from which the third term with negative sign appears. More formally, we apply Berry–Esseen inequality\footnote{see \url{https://en.wikipedia.org/wiki/Berry-Esseen_theorem}} on the rate of approximation in CLT \citep{Shev}:
\begin{equation*}
\left| \Pr\left(\frac{1}{\sigma\sqrt{\tau}}\sum_{i=1}^\tau (X_i-\mu) > t\right) - \Pr\left(N>t\right) \right| \le \frac{1}{2}\frac{\nu^3}{\sigma^3\sqrt{\tau}}, \quad t\in\R,
\end{equation*}
where $N\sim\mathcal{N}(0,1)$ has the standard normal distribution. Setting $t = - \mu\sqrt{\tau}/\sigma$, we get
\begin{equation}\label{CLT-BE}
\left| \Pr\left(\frac{1}{\tau}\sum_{i=1}^\tau X_i > 0\right) - \Pr\left(N > -\frac{\mu\sqrt{\tau}}{\sigma}\right) \right| \le \frac{1}{2}\frac{\nu^3}{\sigma^3\sqrt{\tau}}.
\end{equation}

It remains to compute the second probability using the cumulative distribution function of normal distribuition and express it in terms of the error function:
\begin{align*}
\begin{split}
\Pr\left(\sign\left[\frac{1}{\tau}\sum_{i=1}^{\tau}X_i\right] = \sign \mu \right)
&=   \Pr\left(\frac{1}{\tau}\sum_{i=1}^{\tau}X_i > 0 \right) \\
& \overset{(\ref{CLT-BE})}{\ge} \Pr\left(N > -\frac{\mu\sqrt{\tau}}{\sigma}\right) - \frac{1}{2}\frac{\nu^3}{\sigma^3\sqrt{\tau}} \\
&=   \frac{1}{\sqrt{2\pi}} \int_{-\mu\sqrt{\tau}/\sigma}^{\infty} e^{-t^2/2}\,dt - \frac{1}{2}\frac{\nu^3}{\sigma^3\sqrt{\tau}} \\
&=   \frac{1}{2}\left(1 + \sqrt{\frac{2}{\pi}} \int_{0}^{\mu\sqrt{\tau}/\sigma} e^{-t^2/2}\,dt - \frac{\nu^3}{\sigma^3\sqrt{\tau}}\right) \\
&=   \frac{1}{2}\left(1 + \erf\left(\frac{\mu\sqrt{\tau}}{\sqrt{2}\sigma}\right) - \frac{\nu^3}{\sigma^3\sqrt{\tau}}\right).
\end{split}
\end{align*}
\end{proof}

Clearly, bound (\ref{SPB-nonsmall-entropy}) is better than (\ref{SPB-small-entropy}) when randomness is high. On the other hand, bound (\ref{SPB-nonsmall-entropy}) is not optimal for small randomness ($\sigma\approx 0$). Indeed, one can show that in a small randomness regime, while both variance $\sigma^2$ and third moment $\nu^3$ are small, the ration $\nu/\sigma$ might blow up to infinity producing trivial inequality. For instance, taking $X_i\sim\textrm{Bernoulli}(p)$ and letting $p\to1$ gives $\nu/\sigma = O\left((1-p)^{-\nicefrac{1}{6}}\right)$. This behaviour stems from the fact that we are using CLT: less randomness implies slower rate of approximation in CLT.

As a result of these two bounds on success probabilities, we conclude a condition on mini-batch size for the SPB assumption  to hold.
\begin{lemma}\label{lemma:spb-via-minibatch}
Let $X_1, X_2, \dots, X_{\tau}$ be i.i.d. random variables with non-zero mean $\mu\ne 0$ and finite variance $\sigma^2 < \infty$. Then
\begin{equation}\label{SPB-via-minibatch}
\Pr\left(\sign\left[\frac{1}{\tau}\sum_{i=1}^{\tau}X_i\right] = \sign \mu \right) > \frac{1}{2}, \quad\text{if}\quad
\tau > 2\min \left( \frac{\sigma^2}{\mu^2}, \frac{\nu^3}{|\mu|\sigma^2}\right),
\end{equation}
where $\nu^3$ is (possibly infinite) 3th central moment.
\end{lemma}
\begin{proof}
First, if $\sigma = 0$ then the lemma holds trivially. If $\nu = \infty$, then it follows immediately from Lemma \ref{lemma:SPB-small-entropy}. Assume both $\sigma$ and $\nu$ are positive and finite.

In case of $\tau > 2 \sigma^2/\mu^2$ we apply Lemma \ref{lemma:SPB-small-entropy} again. Consider the case $\tau \le 2\sigma^2/\mu^2$, which implies $\frac{\mu\sqrt{\tau}}{\sqrt{2}\sigma} \le 1$. It is easy to check that $\erf(x)$ is concave on $[0,1]$ (in fact on $[0, \infty)$), therefore $\erf(x) \ge \erf(1)x$ for any $x\in[0,1]$. Setting $x = \frac{\mu\sqrt{\tau}}{\sqrt{2}\sigma}$ we get
\begin{equation*}
\erf\left(\frac{\mu\sqrt{\tau}}{\sqrt{2}\sigma}\right) \ge \frac{\erf(1)}{\sqrt{2}}\frac{\mu\sqrt{\tau}}{\sigma},
\end{equation*}
which together with (\ref{SPB-nonsmall-entropy}) gives
\begin{equation*}
\Pr\left(\sign\left[\frac{1}{\tau}\sum_{i=1}^{\tau}X_i\right] = \sign \mu \right) \ge \frac{1}{2}\left(1 + \frac{\erf(1)}{\sqrt{2}}\frac{\mu\sqrt{\tau}}{\sigma} - \frac{\nu^3}{\sigma^3\sqrt{\tau}}\right).
\end{equation*}
Hence, SPB assumption holds if $$\tau > \frac{\sqrt{2}}{\erf(1)} \frac{\nu^3}{\mu\sigma^2}.$$ It remains to show that $\erf(1)>\nicefrac{1}{\sqrt{2}}$. Convexity of $e^x$ on $x\in[-1,0]$ implies $e^x \ge 1 + (1-\nicefrac{1}{e})x$ for any $x\in[-1,0]$. Therefore
\begin{align*}
\begin{split}
\erf(1)
&=   \frac{2}{\sqrt{\pi}}\int_0^1 e^{-t^2}\,dt \\
&\ge \frac{2}{\sqrt{\pi}}\int_0^1 \left(1 - (1-\nicefrac{1}{e})t^2\right)\,dt \\
&=   \frac{2}{\sqrt{\pi}}\left(\frac{2}{3} + \frac{1}{3e}\right) > \frac{2}{\sqrt{4}}\left(\frac{2}{3} + \frac{1}{3\cdot 3}\right) = \frac{7}{9} > \frac{1}{\sqrt{2}}.
\end{split}
\end{align*}
\end{proof}

Lemma (\ref{lemma:SPB-via-minibatch}) follows from Lemma (\ref{lemma:spb-via-minibatch}) applying it to i.i.d. data $\hat{g}_i^1(x), \hat{g}_i^2(x), \dots, \hat{g}_i^M(x)$.

\subsection{Sufficient conditions for SPB: Proof of Lemma \ref{lemma:SPB-via-entropy}}\label{apx:lemma-SPB-via-entropy}

This observation is followed by the fact that for continuous random variables, the Gaussian distribution has the maximum differential entropy for a given variance\footnote{see \url{https://en.wikipedia.org/wiki/Differential_entropy} or \url{https://en.wikipedia.org/wiki/Normal_distribution\#Maximum_entropy} }. Formally, let $p_G(x)$ be the probability density function (PDF) of a Gaussian random variable with variance $\sigma^2$ and $p(x)$ be the PDF of some random variable with the same variance. Then $H(p) \le H(p_G)$, where
$$
H(p) = - \int_{\R} p(x)\log p(x)\,dx
$$
is the differential entropy of probability distribution $p(x)$ or alternatively differential entropy of random variable with PDF $p(x)$. Differential entropy for normal distribution can be expressed analytically by $H(p_G) = \frac{1}{2}\log(2\pi e \sigma^2)$. Therefore
$$
H(p) \le \frac{1}{2}\log(2\pi e \sigma^2)
$$
for any distribution $p(x)$ with variance $\sigma^2$. Now, under the bounded variance assumption $\E\left[ |\hat{g} - g|^2 \right] \le C$ (where $g$ is the expected value of $\hat{g}$) we have the entropy of random variable $\hat{g}$ bounded by $\frac{1}{2}\log(2\pi e C)$. However, under the SPB assumption $\Pr\left(\sign\hat{g} = \sign g \right) > \nicefrac{1}{2}$ the entropy is unbounded. In order to prove this, it is enough to notice that under SPB assumption random variable $\hat{g}$ could be any Gaussian random variable with mean $g\ne0$. In other words, SPB assumption holds for any Gaussian random variable with non-zero mean. Hence the entropy could be arbitrarily large as there is no restriction on the variance.

\subsection{Convergence analysis for $M=1$: Proof of Theorem \ref{non-convex-theorem}}\label{apx:non-convex-theorem}

Basically, the analysis follows the standard steps used to analyze SGD for non-convex objectives, except the part  (\ref{proof-crossroad-start})--(\ref{proof-crossroad-end}) where inner product $\E[\langle g_k, \sign\hat{g}_k\rangle]$ needs to be estimated. This is exactly the place when stochastic gradient estimator $\sign\hat{g}_k$ interacts with the true gradient $g_k$. In case of standard SGD, we use estimator $\hat{g}_k$ and the mentioned inner product yields $\|g_k\|^2$, which is then used to measure the progress of the method. In our case, we show that
$$
\E[\langle g_k, \sign\hat{g}_k\rangle] = \|g_k\|_{\rho},
$$
with the $\rho$-norm defined in Definition \ref{def:rho-norm}.

Now we present the proof in more details. First, from $L$-smoothness assumption we have
\begin{align*}
\begin{split}
f(x_{k+1})
&= f(x_k - \gamma_k \sign\hat{g}_k) \\
&\le f(x_k) - \langle g_k, \gamma_k\sign\hat{g}_k\rangle + \sum_{i=1}^d \frac{L_i}{2}(\gamma_k\sign\hat{g}_{k,i})^2 \\
&=   f(x_k) - \gamma_k\langle g_k, \sign\hat{g}_k\rangle + \frac{d\bar{L}}{2}\gamma_k^2,
\end{split}
\end{align*}
where $g_k = g(x_k),\, \hat{g}_k = \hat{g}(x_k)$, $\hat{g}_{k,i}$ is the $i$-th component of $\hat{g}_k$ and $\bar{L}$ is the average value of $L_i$'s. Taking conditional expectation given current iteration $x_k$ gives
\begin{equation}\label{1-1}
\E[f(x_{k+1}) | x_k] \le f(x_k) - \gamma_k\E[\langle g_k, \sign\hat{g}_k\rangle] + \frac{d\bar{L}}{2}\gamma_k^2.
\end{equation}
Using the definition of success probabilities $\rho_i$ we get
\begin{align}
\E[\langle g_k, \sign\hat{g}_k\rangle] \label{proof-crossroad-start}
&= \langle g_k, \E[\sign\hat{g}_k]\rangle \\
&= \sum_{i=1}^d g_{k,i} \cdot \E[\sign\hat{g}_{k,i}]
= \sum_{\substack{1\le i\le d \\ g_{k,i}\ne0}} g_{k,i} \cdot \E[\sign\hat{g}_{k,i}] \\
&= \sum_{\substack{1\le i\le d \\ g_{k,i}\ne0}} g_{k,i} \left(\rho_i(x_k)\sign g_{k,i} + (1-\rho_i(x_k))(-\sign g_{k,i}) \right) \\
&= \sum_{\substack{1\le i\le d \\ g_{k,i}\ne0}} (2\rho_i(x_k) - 1)|g_{k,i}| = \sum_{i=1}^d (2\rho_i(x_k) - 1)|g_{k,i}| = \|g_k\|_\rho. \label{proof-crossroad-end}
\end{align}

Plugging this into (\ref{1-1}) and taking full expectation, we get
\begin{equation}\label{1-2}
\E\|g_k\|_\rho \le \frac{\E[f(x_k)]-\E[f(x_{k+1})]}{\gamma_k} + \frac{d\bar{L}}{2}\gamma_k.
\end{equation}
Therefore
\begin{equation}\label{1-3}
\sum_{k=0}^{K-1} \gamma_k \E\|g_k\|_\rho \le (f(x_0) - f^*) + \frac{d\bar{L}}{2}\sum_{k=0}^{K-1}\gamma_k^2.
\end{equation}
Now, in case of decreasing step sizes $\gamma_k = \gamma_0/\sqrt{k+1}$
\begin{align*}
\min_{0\le k < K}\E\|g_k\|_\rho
&\le \sum_{k=0}^{K-1} \frac{\gamma_0}{\sqrt{k+1}} \E\|g_k\|_\rho \bigg/ \sum_{k=0}^{K-1} \frac{\gamma_0}{\sqrt{k+1}} \\
&\le \frac{1}{\sqrt{K}} \left[\frac{f(x_0) - f^*}{\gamma_0} + \frac{d\bar{L}}{2}\gamma_0\sum_{k=0}^{K-1}\frac{1}{k+1}\right] \\
&\le \frac{1}{\sqrt{K}} \left[\frac{f(x_0) - f^*}{\gamma_0} + \gamma_0 d\bar{L} + \frac{\gamma_0 d\bar{L}}{2}\log{K}\right] \\
& =  \frac{1}{\sqrt{K}} \left[\frac{f(x_0) - f^*}{\gamma_0} + \gamma_0 d\bar{L}\right] + \frac{\gamma_0 d\bar{L}}{2}\frac{\log{K}}{\sqrt{K}}.
\end{align*}
where we have used the following standard inequalities
\begin{equation}\label{0-1}
\sum_{k=1}^K \frac{1}{\sqrt{k}} \ge \sqrt{K}, \quad \sum_{k=1}^K \frac{1}{k} \le 2 + \log{K}.
\end{equation}

In the case of constant step size $\gamma_k = \gamma$
\begin{equation*}
\frac{1}{K}\sum_{k=0}^{K-1}\E\|g_k\|_\rho \le \frac{1}{\gamma K}\left[ (f(x_0) - f^*) + \frac{d\bar{L}}{2}\gamma^2 K \right] = \frac{f(x_0) - f^*}{\gamma K} + \frac{d\bar{L}}{2}\gamma.
\end{equation*}

\subsection{Convergence analysis for $M=1$: Proof of Theorem \ref{non-convex-theorem-2}}\label{apx:non-convex-theorem-2}

Clearly, the iterations $\{x_k\}_{k\ge 0}$ of Algorithm \ref{alg:signSGD} with Option 2 do not increase the function value in any iteration, i.e. $\E[f(x_{k+1})|x_k] \le f(x_k)$. Continuing the proof of Theorem \ref{non-convex-theorem} from (\ref{1-2}), we get
\begin{align*}
\frac{1}{K}\sum_{k=0}^{K-1} \E\|g_k\|_\rho
&\le \frac{1}{K}\sum_{k=0}^{K-1} \frac{\E[f(x_k)]-\E[f(x_{k+1})]}{\gamma_k} + \frac{d\bar{L}}{2}\gamma_k \\
& =  \frac{1}{K}\sum_{k=0}^{K-1} \frac{\E[f(x_k)]-\E[f(x_{k+1})]}{\gamma_0} \sqrt{k+1} + \frac{d\bar{L}}{2K}\sum_{k=0}^{K-1}\frac{\gamma_0}{\sqrt{k+1}} \\
&\le \frac{1}{\sqrt{K}}\sum_{k=0}^{K-1} \frac{\E[f(x_k)]-\E[f(x_{k+1})]}{\gamma_0} + \frac{\gamma_0 d\bar{L}}{\sqrt{K}} \\
& =  \frac{f(x_0)-\E[f(x_K)]}{\gamma_0\sqrt{K}} + \frac{\gamma_0 d\bar{L}}{\sqrt{K}} \\
&\le \frac{1}{\sqrt{K}} \left[\frac{f(x_0)-f^*}{\gamma_0} + \gamma_0 d\bar{L}\right],
\end{align*}
where we have used the following inequality
\begin{equation*}
\sum_{k=1}^K \frac{1}{\sqrt{k}} \le 2\sqrt{K}.
\end{equation*}
The proof for constant step size is the same as in Theorem \ref{non-convex-theorem}.

\subsection{Convergence analysis in parallel setting: Proof of Theorem \ref{non-convex-theorem-d}}\label{apx:non-convex-theorem-d}

First, denote by $I(p;a,b)$ the regularized incomplete beta function, which is defined as follows
\begin{equation}\label{def:ri-beta}
I(p;a,b) = \frac{B(p;a,b)}{B(a,b)} = \frac{\int_0^p t^{a-1}(1-t)^{b-1}\,dt}{\int_0^1 t^{a-1}(1-t)^{b-1}\,dt}, \quad a,b > 0,\, p\in[0,1].
\end{equation}

The proof of Theorem \ref{non-convex-theorem-d} goes with the same steps as in Theorem \ref{non-convex-theorem}, except the derivation (\ref{proof-crossroad-start})--(\ref{proof-crossroad-end}) is replaced by
\begin{align*}
\E[\langle g_k, \sign\hat{g}_k^{(M)}\rangle]
&= \langle g_k, \E[\sign\hat{g}_k^{(M)}]\rangle \\
&= \sum_{i=1}^d g_{k,i} \cdot \E[\sign\hat{g}_{k,i}^{(M)}] \\
&= \sum_{\substack{1\le i\le d \\ g_{k,i}\ne0}} |g_{k,i}| \cdot \E\left[\sign\left(\hat{g}^{(M)}_{k,i}\cdot g_{k,i}\right)\right] \\
&= \sum_{\substack{1\le i\le d \\ g_{k,i}\ne0}} |g_{k,i}| \left(2 I(\rho_i(x_k); l, l) - 1\right) = \|g_k\|_{\rho_M},
\end{align*}
where we have used the following lemma.

\begin{lemma}
Assume that for some point $x\in\R^d$ and some coordinate $i\in\{1, 2, \dots, d\}$, master node receives $M$ independent stochastic signs $\sign\hat{g}_i^n(x),\,n=1,\dots,M$ of true gradient $g_i(x)\ne0$. Let $\hat{g}^{(M)}(x)$ be the sum of stochastic signs aggregated from nodes:
\begin{equation*}
\hat{g}^{(M)} = \sum_{n=1}^M\sign\hat{g}^n.
\end{equation*}
Then
\begin{equation}\label{master-sign-exp}
\E\left[\sign\left(\hat{g}^{(M)}_i\cdot g_i\right)\right]  = 2 I(\rho_i; l, l) - 1,
\end{equation}
where $l = \lfloor\tfrac{M+1}{2}\rfloor$ and $\rho_i>\nicefrac{1}{2}$ is the success probablity for coordinate $i$.
\end{lemma}

\begin{proof}
Denote by $S_i^n$ the Bernoulli trial of node $n$ corresponding to $i$th coordinate, where ``success'' is the sign match between stochastic gradient and gradient:
\begin{equation}\label{def:Sim}
S^n_i :=
\begin{cases}
        1, & \text{if } \sign\hat{g}^n_i = \sign g_i\\
        0, & \text{otherwise}
\end{cases}
\sim \textrm{Bernoulli}(\rho_i).
\end{equation}
Since nodes have their own independent stochastic gradients and the objective function (or dataset) is shared, then master node receives i.i.d. trials $S_i^n$, which sum up to a binomial random variable $S_i$:
\begin{equation}\label{def:Si}
S_i := \sum_{n=1}^M S_i^n \sim \textrm{Binomial}(M, \rho_i).
\end{equation}
First, let us consider the case when there are odd number of nodes, i.e. $M=2l-1,\,l\ge 1$. In this case, taking into account (\ref{def:Sim}) and (\ref{def:Si}), we have
\begin{align*}
                &\,\Pr\left(\sign\hat{g}^{(M)}_i = 0\right) = 0, \\
\rho_i^{(M)}    := &\,\Pr\left(\sign\hat{g}^{(M)}_i = \sign g_i\right) = \Pr(S_i \ge l), \\
1 - \rho_i^{(M)} = &\,\Pr\left(\sign\hat{g}^{(M)}_i = -\sign g_i\right).
\end{align*}

It is well known that cumulative distribution function of binomial random variable can be expressed with regularized incomplete beta function:
\begin{equation}\label{binom-cdf}
\Pr(S_i \ge l) = I(\rho_i; l, M-l+1) = I(\rho_i; l, l).
\end{equation}

Therefore,
\begin{align*}
\E\left[\sign\left(\hat{g}^{(M)}_i\cdot g_i\right)\right]
&=   \rho_i^{(M)}\cdot 1 + (1-\rho_i^{(M)})\cdot(-1) \\
&=   2\rho_i^{(M)} - 1 \\
&=   2 \Pr(S_i \ge l) - 1 \\
&=   2 I(\rho_i; l, l) - 1.
\end{align*}

In the case of even number of nodes, i.e. $M=2l,\,l\ge 1$, there is a probability to fail the vote $\Pr\left(\sign\hat{g}^{(M)}_i = 0\right) > 0$. However using (\ref{binom-cdf}) and properties of beta function\footnote{see \url{https://en.wikipedia.org/wiki/Beta_function\#Incomplete_beta_function}} gives
\begin{align*}
\E\left[\sign\left(\hat{g}^{(2l)}_i\cdot g_i\right)\right]
&= \Pr(S_i \ge l+1) \cdot 1 + \Pr(S_i \le l-1) \cdot (-1) \\
&= I(\rho_i;l+1,l) + I(\rho_i;l,l+1) - 1 \\
&= 2 I(\rho_i;l,l) - 1 \\
&= \E\left[\sign\left(\hat{g}^{(2l-1)}_i\cdot g_i\right)\right].
\end{align*}
This also shows that in expectation there is no difference between having $2l-1$ and $2l$ nodes.
\end{proof}

\subsection{Convergence analysis in parallel setting: Speedup with respect to $M$}\label{apx:exp-vr}

Here we present the proof of exponential noise reduction in parallel setting in terms of number of nodes. We first state the well-known Hoeffding's inequality:
\begin{theorem}[Hoeffding's inequality for general bounded random variables; see \citep{Versh}, Theorem 2.2.6]
Let $X_1, X_2, \dots, X_M$ be independent random variables. Assume that $X_n \in [A_n, B_n]$ for every $n$. Then, for any t > 0, we have
\begin{equation*}
\Pr\left(\sum_{n=1}^M \left(X_n - \E X_n \right) \ge t\right) \le \exp\left(-\frac{2t^2}{\sum_{n=1}^M (B_n-A_n)^2}\right).
\end{equation*}
\end{theorem}

Define random variables $X_i^n,\,n=1, 2, \dots, M$ showing the missmatch between stochastic gradient sign and full gradient sign  from node $n$ and coordinate $i$:
\begin{equation}\label{def:Xim}
X^n_i :=
\begin{cases}
        -1, & \text{if } \sign\hat{g}^n_i = \sign g_i\\
        1, & \text{otherwise}
\end{cases}
\end{equation}

Clearly $\E X_i^n = 1-2\rho_i$ and Hoeffding's inequality gives
\begin{equation*}
\Pr\left(\sum_{n=1}^M X_i^n - M(1-2\rho_i) \ge t\right) \le \exp\left(-\frac{t^2}{2M}\right), \quad t>0.
\end{equation*}
Choosing $t = M(2\rho_i-1)>0$ (because of SPB assumption) yields
\begin{equation*}
\Pr\left(\sum_{n=1}^M X_i^n \ge 0\right) \le \exp\left(-\frac{1}{2}(2\rho_i-1)^2M\right).
\end{equation*}

Using Lemma \ref{master-sign-exp}, we get
\begin{equation*}
2 I(\rho_i,l;l) - 1 = \E\left[\sign\left(\hat{g}^{(M)}_i\cdot g_i\right)\right] = 1 - \Pr\left(\sum_{n=1}^M X_i^n \ge 0\right) \ge 1 - \exp\left(-(2\rho_i-1)^2 l\right),
\end{equation*}
which provides the following estimate for $\rho_M$-norm:
\begin{equation*}
\left(1-\exp\left(-(2\rho(x)-1)^2 l\right)\right)\|g(x)\|_1 \le \|g(x)\|_{\rho_M} \le \|g(x)\|_1,
\end{equation*}
where $\rho(x) = \min_{1\le i\le d}\rho_i(x) > \frac{1}{2}$.

\subsection{Distributed training in the homogeneous setting: Proof of Theorem \ref{thm:ssdm}}\label{apx:ssdm}

We follow the analysis of \citet{N-SGD-M}, who derived similar convergence rate for normalized SGD in single node setting. The novelty in our proof technique is i) extending the analysis in distributed setting and ii) establishing a connection between normalized SGD and sign-based methods via the new notion of {\em stochastic sign}.

\begin{lemma}[see Lemma 2 in \citep{N-SGD-M}]\label{lem:inner-prod-bound} For any non-zero vectors $a$ and $b$
$$-\frac{\langle a,b\rangle}{\|a\|} \le -\frac{1}{3}\|b\| + \frac{8}{3}\|a-b\|.$$
\end{lemma}
\begin{proof}
Denote $c=a-b$ and consider two cases. If $\|c\|\le \frac{1}{2}\|b\|$, then
$$
-\frac{\langle a,b\rangle}{\|a\|} = -\frac{\|b\|^2 + \langle c,b\rangle}{\|a\|} \le -\frac{\|b\|^2 - \|c\|\|b\|}{\|b+c\|} \le -\frac{\|b\|^2}{2\|b+c\|} \le -\frac{1}{3}\|b\| \le -\frac{1}{3}\|b\| + \frac{8}{3}\|a-b\|.
$$
Alternatively, if $\|c\| > \frac{1}{2}\|b\|$, then
$$
-\frac{\langle a,b\rangle}{\|a\|} \le \|b\| \le -\frac{1}{3}\|b\| + \frac{4}{3}\|b\| \le -\frac{1}{3}\|b\| + \frac{8}{3}\|c\|.
$$
\end{proof}

We start from the smoothness of function $f$:
\begin{align*}
f(x_{k+1})
&\le f(x_k) - \frac{\gamma}{M} \langle \nabla f(x_k), s_k\rangle + \frac{L\gamma^2}{2M^2}\|s_k\|^2 \\
&=   f(x_k) - \frac{\gamma}{M} \langle \nabla f(x_k), s_k\rangle + \frac{L\gamma^2}{2} \left\|\frac{1}{M}\sum_{n=1}^M s_k^n\right\|^2 \\
&\le f(x_k) - \frac{\gamma}{M} \langle \nabla f(x_k), s_k\rangle + \frac{L\gamma^2}{2} \frac{1}{M}\sum_{n=1}^M\left\|s_k^n\right\|^2 \\
&=   f(x_k) - \frac{\gamma}{M}\sum_{n=1}^M \langle \nabla f(x_k), s_k^n\rangle + \frac{d L\gamma^2}{2}.
\end{align*}

Denote $g_k = \nabla f(x_k)$. Taking expectation conditioned on previous iterate $x_k$ and current stochastic gradient $\hat{g}_k^n$, we get
\begin{align}
\begin{split}\label{bound-01}
\E\left[ f(x_{k+1}) | x_k, \hat{g}_k^n \right]
&\le f(x_k) - \frac{\gamma}{M}\sum_{n=1}^M \frac{\langle g_k, m_k^n\rangle}{\|m_k^n\|} + \frac{d L\gamma^2}{2} \\
&\overset{\textrm{Lemma}\;\ref{lem:inner-prod-bound}}{\le} f(x_k) - \frac{\gamma}{3}\|g_k\| + \frac{8\gamma}{3M}\sum_{n=1}^M \|m_k^n-g_k\| + \frac{d L\gamma^2}{2}.
\end{split}
\end{align}

Next, we find recurrence relation for the error terms $\hat{\epsilon}_k^n \eqdef m_k^n-g_k$. Denote $\epsilon_k^n \eqdef \hat{g}_k^n-g_k$, and notice that
\begin{align*}
\hat{\epsilon}_{k+1}^n
&= \beta m_k^n + (1-\beta)\hat{g}_{k+1}^n - g_{k+1} \\
&= \beta (m_k^n - g_{k+1}) + (1-\beta)(\hat{g}_{k+1}^n - g_{k+1}) \\
&= \beta(m_k^n-g_k) + \beta(g_k - g_{k+1}) + (1-\beta)\epsilon_{k+1}^n \\
&= \beta\hat{\epsilon}_k^n + \beta(g_k - g_{k+1}) + (1-\beta)\epsilon_{k+1}^n.
\end{align*}

Unrolling this recursion and noting that $\hat{\epsilon}_0^n = \epsilon_0^n$ (due to initial moment $m_{-1}^n = \hat{g}_0^n$), we get
$$
\hat{\epsilon}_{k+1}^n = \beta^{k+1}\epsilon_0^n + \beta\sum_{t=0}^k \beta^t(g_{k-t}-g_{k+1-t}) + (1-\beta)\sum_{t=0}^k \beta^t\epsilon_{k+1-t}^n
$$

From Assumption \ref{asum-BV}, we have
\begin{equation}\label{error-ctr}
\E\left[\langle \epsilon_k^n, \epsilon_{k'}^n\rangle\right]
\begin{cases}
\le(\sigma^n)^2 & \text{ if } k=k', \\
=0 & \text{ if } k\ne k'.
\end{cases}
\end{equation}

Using $L$-smoothness of $f$, we have
\begin{equation}\label{L_n-smooth-bound}
\|g_k - g_{k+1}\| \le L\|x_k - x_{k+1}\| = \frac{L\gamma}{M}\|s_k\| \le L\gamma\sqrt{d}.
\end{equation}

Therefore
\begin{align*}
\E\|\hat{\epsilon}_{k+1}^n\|
&\le \beta^{k+1}\E\|\epsilon_0^n\| + \sum_{t=0}^k \beta^{t+1}\|g_{k-t}-g_{k+1-t}\| + (1-\beta)\E\left\|\sum_{t=0}^k \beta^t\epsilon_{k+1-t}^n\right\| \\
&\overset{(\ref{L_n-smooth-bound})}{\le} \beta^{k+1}\sigma^n + \frac{L\gamma\sqrt{d}}{1-\beta} + (1-\beta)\sqrt{\E\left\|\sum_{t=0}^k \beta^t\epsilon_{k+1-t}^n\right\|^2} \\
&\overset{(\ref{error-ctr})}{\le} \beta^{k+1}\sigma^n + \frac{L\gamma\sqrt{d}}{1-\beta} + (1-\beta)\sqrt{\sum_{t=0}^k \beta^{2t}(\sigma^n)^2} \\
&\le \beta^{k+1}\sigma^n + \frac{L\gamma\sqrt{d}}{1-\beta} + \sigma^n\sqrt{1-\beta} \\
\end{align*}
Averaging this bound over the nodes yields
$$
\frac{1}{M}\sum_{n=1}^M\E\|\hat{\epsilon}_k^n\| \le \beta^k\tilde{\sigma} + \frac{L\gamma\sqrt{d}}{1-\beta} + \tilde{\sigma}\sqrt{1-\beta}.
$$
Then averaging over the iterates gives
$$
\frac{1}{KM}\sum_{k=0}^{K-1}\sum_{n=1}^M\E\|\hat{\epsilon}_k^n\| \le \frac{\tilde{\sigma}}{(1-\beta)K} + \frac{L\gamma\sqrt{d}}{1-\beta} + \tilde{\sigma}\sqrt{1-\beta}.
$$

Taking full expectation in (\ref{bound-01}), we have
\begin{align*}
\frac{1}{K}\sum_{k=0}^{K-1}\E\|g_k\|
&\le \frac{3}{\gamma K} \sum_{k=0}^{K-1}\E\left[f(x_k)-f(x_{k+1})\right] + \frac{8}{MK}\sum_{n=1}^{M}\sum_{k=0}^{K-1}\E\|\hat{\epsilon}_k^n\| + \frac{3}{2}Ld\gamma \\
&\le \frac{3(f(x_0)-f_*)}{\gamma K} + \frac{8\tilde{\sigma}}{(1-\beta)K} + \frac{8L\gamma\sqrt{d}}{1-\beta} + 8\tilde{\sigma}\sqrt{1-\beta} + \frac{3}{2}Ld\gamma.
\end{align*}

Now it remains to choose parameters $\gamma$ and $\beta$ properly. Setting $\beta = 1-\frac{1}{\sqrt{K}}$ and $\gamma = \frac{1}{K^{\nicefrac{3}{4}}}$, we get
\begin{align*}
\frac{1}{K}\sum_{k=0}^{K-1}\E\|g_k\|
&\le \frac{3\delta_f}{\gamma K} + \frac{8\tilde{\sigma}}{(1-\beta)K} + \frac{8L\gamma\sqrt{d}}{1-\beta} + 8\tilde{\sigma}\sqrt{1-\beta} + 3Ld\gamma \\
&\le \frac{1}{K^{\nicefrac{1}{4}}}\left[ 3\delta_f + 16\tilde{\sigma} + 8L\sqrt{d} + \frac{3Ld}{\sqrt{K}} \right].
\end{align*}

\section{Recovering Theorem 1 in \citep{BZAA} from Theorem~\ref{non-convex-theorem}} \label{sec:recovering}

To recover Theorem 1 in \citep{BZAA}, first note that choosing a particular step size $\gamma$ in (\ref{const-step-rate}) yields
\begin{equation}\label{const-step-rate-K}
\frac{1}{K}\sum_{k=0}^{K-1}\E\|g_k\|_\rho \le \sqrt{\frac{2 d\bar{L} (f(x_0)-f^*)}{K}}, \quad\text{with}\quad \gamma = \sqrt{\frac{2(f(x_0)-f^*)}{d\bar{L}K}}.
\end{equation}
Then, due to Lemma \ref{lemma:main-rho}, under unbiasedness and unimodal symmetric noise assumption, we can lower bound general $\rho$-norm by mixed $l^{1,2}$ norm. Finally we further lower bound our $l^{1,2}$ norm to obtain \emph{the mixed norm} used in Theorem 1 of \citep{BZAA}:
\begin{align*}
5\sqrt{\frac{d\bar{L} (f(x_0)-f^*)}{K}} 
&\ge \frac{5}{\sqrt{2}} \frac{1}{K}\sum_{k=0}^{K-1}\E\|g_k\|_\rho \\
&\ge \frac{5}{\sqrt{2}} \frac{1}{K}\sum_{k=0}^{K-1}\E\|g_k\|_{l^{1,2}}
=    \frac{5}{\sqrt{2}} \frac{1}{K}\sum_{k=0}^{K-1}  \left[\sum_{i=1}^d \frac{g_i^2}{|g_i| + \sqrt{3}\sigma_i}\right] \\
&\ge \frac{5}{\sqrt{2}} \frac{1}{K}\sum_{k=0}^{K-1}\E\left[\frac{2}{5}\sum_{i\in H_k}|g_{k,i}| + \frac{\sqrt{3}}{5}\sum_{i\notin H_k}\frac{g_{k,i}^2}{\sigma_i}\right] \\
&\ge \frac{1}{K}\sum_{k=0}^{K-1}\E\left[\sum_{i\in H_k}|g_{k,i}| + \sum_{i\notin H_k}\frac{g_{k,i}^2}{\sigma_i}\right],
\end{align*}
where $H_k = \{1\le i\le d \colon \sigma_i<\frac{\sqrt{3}}{2}|g_{k,i}|\}$.


\end{document}